\documentclass[11pt,reqno]{amsart}
\usepackage{xifthen}
\usepackage{subfig}
\usepackage{pkgfile}

\newcommand{\Graph}{{\mathsf G}}
    \DeclareMathOperator{\sech}{sech}
\newcommand{\norm}[1]{\left \| #1 \right \|}
\begin{document}

\title[Mean-field Potts model]{Universality of the mean-field for the Potts model}

\author[A.\ Basak]{Anirban Basak$^*$}
 \author[S.\ Mukherjee]{Sumit Mukherjee$^\dagger$}
 
 \address{$^*$Department of Mathematics, Duke University
\newline\indent Durham, North Carolina 27705}

\address{$^{\dagger}$Department of Statistics, Columbia University
\newline\indent Manhattan, New York 10027}

\date{\today}




\begin{abstract}
We consider the Potts model with $q$ colors on a sequence of weighted graphs with adjacency matrices  $A_n$, allowing for both positive and negative weights. Under a mild regularity condition on $A_n$ we show that the mean-field prediction for the log partition function is asymptotically correct, whenever $\tr(A_n^2)=o(n)$. In particular, our results are applicable for the Ising and the Potts models on any sequence of graphs with average degree going to $+\infty$. Using this, we establish the universality of  the limiting log partition function of the ferromagnetic Potts model for a  sequence of asymptotically regular graphs, and that of the Ising model for bi-regular bipartite graphs in both ferromagnetic and anti-ferromagnetic domain. We also derive a large deviation principle for the empirical measure of the colors for the Potts model on asymptotically regular graphs.
	
\end{abstract}

\maketitle

\section{Introduction}
One of the fundamental models in statistical physics is the nearest neighbor {\em $q$-state  Potts model}. For a finite undirected graph $\Graph:=(V, E)$, with vertex set $V$, and edge set $E$, the Potts model is a probability measure on $[q]^{|V|}$ with $[q]:=\{1,2,\cdots,q\}$, where $|\cdot|$ denotes the cardinality of a set. The probability mass function for the Potts model at ${\bm y}:= \{y_i, i \in V\}$ is given by
\beq\label{eq:basic_potts}
\mu^{\be, B}({\bm y}):= \f{1}{Z_\Graph(\be, B)} \exp \Big\{ \be \sum_{(i,j) \in E} \delta(y_i,y_j) + B \sum_{i \in V} \delta(y_i,1)\Big\}.
\eeq
Here $\delta(y,y')= 1_{y=y'}$, and $Z_\Graph(\be, B)$ is the normalizing constant, which is commonly termed as the {\em partition function}. The parameters $\be$ and $B$ are known as {\em inverse temperature} parameter and {\em external magnetic field} parameters respectively, with $\be \ge 0$ is said to be the {\em ferromagnetic} regime, and $\be <0$ is the {\em anti-ferromagnetic} regime. When $q=2$, the measure $\mu^{\be, B}(\cdot)$ is the well known {\em Ising measure}. 

Although Ising and Potts models originated from statistical physics \cite{Ising,Potts}, due to its wide applications it has received a lot of recent interest from varied areas, including  statistics (cf.~\cite{structure_learning,BM,Chatterjee,highdim_ising} and references therein), computer science (cf.~\cite{Ben-Mon,bresler,Gar,Sly-Sun} and references therein), combinatorics,  finance,
social networks, computer vision, biology, and signal processing. Potts models on graphs also have connections with many graph properties, such as the number of proper colorings, max cut, min cut, min bisection (cf.~\cite{bgt,graph_limits_I,graph_limits_II,dmsen} and references therein), which are of interest in classical graph theory.
One of the main difficulties in the study of the Ising and the Potts model is the intractability of its partition function. If the partition function were available in closed form, one could analyze it to compute moments and limiting distributions, carry on inference in a statistical framework using maximum likelihood, or compute thermodynamic limits of these models which are of interest in statistical physics. As the partition function involves summing the unnormalized mass function over exponentially many terms, computing the partition function numerically or otherwise is challenging in general.  Since exact computations are infeasible, they are broadly two approaches to tackle this problem. A branch of research is directed towards devising efficient algorithms to approximate the log partition function (cf.~\cite{JS,MJW}, and the references therein). Whereas, probabilists are interested in studying the  asymptotics of the log partition function for sequence of graphs $\Graph_n$ for large $n$ (cf.~\cite{dms,dmss,en,ew} and references therein), in an attempt to  understand these measures. More precisely, considering a sequence of graphs $\Graph_n:=([n], E_n)$, with growing size, the goal is to compute the asymptotic {\em limiting log partition function} $\Phi(\be, B)$, where
\[
\Phi(\be, B) := \lim_{n \ra \infty}\frac{1}{n}\Phi_n(\be, B),
\] 
and $\Phi_n(\be, B):= \log Z_{\Graph_n} (\be,B)$. To get a non-trivial value of $\Phi(\be, B)$ one must scale $\be$ appropriately depending on $|E_n|$. In particular, the inverse temperature parameter in \eqref{eq:basic_potts} should be replaced by $\be_n:=({n}/{2|E_n|})\be$ for the Potts model on $\Graph_n$. This scaling ensures that $\Phi(\be,B)$ is not a constant function for all choices of $\be$, and $B$. By a slight abuse of notation we denote this measure by $\mu_n^{\be, B}(\cdot)$. 

One common scheme of approximating  $\Phi_n(\be,B)$ is via the {\em naive mean-field} method. Mean-field method has been in the statistical physics literature for a long time (see \cite{chandler, parisi}). Below we describe the mean-field method in our context in detail:

\subsection{Mean-field method}
Let $\cP([q]^n)$ denote the space of probability measures on $[q]^n$. For any two measures $\mu,\nu\in \cP([q]^n)$ define the Kullback-Leibler divergence between $\mu$ and $\nu$ by
$$D(\mu||\nu):=\sum_{{\bm y}\in [q]^n} \mu({\bm y})\log \mu({\bm y})-\sum_{{\bm y}\in [q]^n}\mu({\bm y})\log {\nu({\bm y})},$$ where $0\log 0=0$ and $\log 0=-\infty$ by convention.

Then,  for any $\gq\in\cP([q]^n)$ an easy computation gives
\begin{align*}
	D(\gq||\mu_n^{\be, B})=\Phi_n(\be, B)+\sum_{{\bm y}\in [q]^n}\gq({\bm y})\log \gq({\bm y})-\sum_{{\bm y}\in [q]^n}\gq({\bm y}){\bm H}_{n}^{\be, B}({\bm y}),
	\end{align*}
	where
	\[
	{\bm H}_{n}^{\be, B}({\bm y}):= \be_n \sum_{(i,j) \in E_n} \delta(y_i,y_j) + B \sum_{i \in [n]} \delta(y_i,1).
	\]
Since $D(\gq||\mu_n^{\be, B})\ge 0$, with equality iff $\gq=\mu_n^{\be, B}$, we get 

\beq\label{eq:var_formula_for_graph}
\Phi_n(\be, B)=\sup_{\gq\in \mathcal{P}([q]^n)} \left\{\sum_{{\bm y}\in [q]^n}\gq({\bm y}){\bm H}_{n}^{\be, B}({\bm y})-\sum_{{\bm y}\in [q]^n}\gq({\bm y})\log \gq({\bm y})\right\}.
\eeq	
 In literature \eqref{eq:var_formula_for_graph} is known as the {\em variational formula} for the log partition function $\Phi_n(\be, B)$. From \eqref{eq:var_formula_for_graph} one can obtain a lower bound on $\Phi_n(\be, B)$ by restricting the supremum in \eqref{eq:var_formula_for_graph} to product measures, i.e.  $\gq=\prod_{i\in [n]} \gq_i \in \mathcal{P}([q])^n$. Therefore
\begin{align}\label{eq:var_mean_field_prediction_for_graph}
	\Phi_n(\be, B)\
\ge	\sup_{\gq\in\cP([q])^n}{\mathsf M}_n^{\be, B}(\gq),
\end{align}
where 
\begin{align}\label{eq:def_mean_field_for_graph}
{\mathsf M}_n^{\be, B}(\gq):=\left\{\be_n {\sum_{(i,j)\in E_n}}\sum_{r\in [q]} \gq_i(r)\gq_j(r)+B\sum_{i\in[n]}\gq_i(1)-\sum_{i\in [n],r\in [q]}\gq_i(r)\log \gq_i(r)\right\}.
\end{align}
The RHS of \eqref{eq:var_mean_field_prediction_for_graph} is referred as the {\em mean-field approximation} for the log-partition function $\Phi_n(\be, B)$.

Since the supremum in \eqref{eq:var_mean_field_prediction_for_graph} is much more tractable than the one in \eqref{eq:var_formula_for_graph}, it is therefore naturally interesting to find graph sequences for which \eqref{eq:var_mean_field_prediction_for_graph} is asymptotically tight. For the complete graph it has been long known that the mean-field prediction is indeed tight for both Ising and Potts measure (see \cite{en,ew,ew_mle}). 
However, for {\em locally tree-like graphs} (see \cite[Definition 1.1]{dms}) this is not the case. Indeed, in \cite{dmaop} it is shown that the {\em Bethe prediction} is the correct answer for Ising measures on such graphs when the limiting tree is a Galton-Watson tree whose off-spring distribution have a finite variance. In \cite{DGH} it was extended for power law distribution, and finally in \cite{dms} it was extended to full generality. Moreover the same was shown be true for the Potts model on regular graphs in \cite{dms,dmss}.

For the complete graph on $n$ vertices one has $\Theta(n^2)$ edges, whereas locally tree-like graphs has only $O(n)$ edges (see Definition \ref{defn:order} for $O(\cdot)$, and $\Theta(\cdot)$). Therefore, it is natural to ask for graph sequences such that $n \ll |E_n| \ll n^2$, if one of the two predictions is correct for the limiting log partition function. Very few results are known about the asymptotics of the log partition function in this regime. See however \cite[Theorem 2.10]{sparse2} which in particular shows that if a sequence of graphs converges in $L^p$ cut metric, then corresponding log partition functions converge.  Also, it follows from \cite[Theorem 1]{CP} that the mean field approximation is correct for the limiting log-partition function of Potts models on a sequence of growing graphs in $\Z^d$, when $d$ goes to $\infty$ as well. We re-derive both these results to demonstrate flexibility of our approach (see Theorem \ref{thm:lp} and Example \ref{eg}(d) respectively).

In this paper, we consider Ising and Potts measures (we consider a slightly generalized version of standard Potts model, see Definition \ref{defn:potts_general}) on graphs with growing sizes such that $|E_n|/n \ra \infty$, as $n \ra \infty$, and show that the {asymptotic log partition function} can be expressed as a {\em variational problem} (see Theorem \ref{defn:potts_general}). Building on Theorem \ref{defn:potts_general}, and focusing on {\em asymptotically regular graphs}, we prove the universality of the limiting log partition function in the ferromagnetic domain, and confirm that it matches with the one obtained from the complete graph (see Theorem \ref{thm:mean_field}). We further derive asymptotic {log partition function} for bi-regular bipartite graphs (see Theorem \ref{thm:bipartite}). Recently, in \cite{sparse2} the asymptotic {log partition function} was derived for graph sequences converging in {\em cut metric}. As a byproduct of Theorem \ref{thm:main} we give an alternate proof of the same (see Section \ref{sec:cut metric}). For an outline of the proof techniques of the results we refer the reader to Section \ref{sec:proof technique}.

%

\subsection{Statement of main theorem}

We will work with the following slightly general version of the Potts model.
\begin{defn}\label{defn:potts_general}
For $q \ge 2$, let ${\bm J}$, ${\bm h}$ be a symmetric $q\times q$ matrix, and a vector of length $q$ respectively. Also let $A_n$ be a real symmetric $n \times n$ matrix. We define a hamiltonian ${\bm H}_n^{{\bm J}, {\bm h}}(\cdot)$ on $[q]^n$ by setting 
\beq\label{eq:hamiltonian}
{\bm H}_n^{{\bm J}, {\bm h}}({\bm y}):=\frac{1}{2}\sum_{i,j=1}^nA_n(i,j)\sum_{r,s=1}^ qJ_{rs} \delta(y_i,r)\delta(y_j,s)+\sum_{i=1}^n \sum_{r=1}^qh_r\delta(y_i,r),
\eeq
where ${\bm y}:=(y_1,\ldots,y_n)$. Using ${\bm H}_n^{{\bm J},{\bm h}}(.)$ we now define the following probability measure on $[q]^n$:
\beq\label{eq:potts_general}
\mu_n^{{\bm J}, {\bm h}}({\bm y}):= \f{1}{Z_n({\bm J},{\bm h})}\exp({\bm H}_n^{{\bm J}, {\bm h}}({\bm y})),
\eeq
where
\begin{align}
	Z_n({\bm J},{\bm h}):= \sum_{{\bm y}\in [q]^n} e^{{\bm H}_{n}^{{\bm J}, {\bm h}}({\bm y})}.  \notag
\end{align}
\end{defn}

\medskip

\noindent
Considering ${\bm J}$ to be the identity matrix $I_q$, ${\bm h}=B(1,0,0,\ldots,0)$, and $A_n$ to be the adjacency matrix of ${\sf G}_n$ divided by $2|E_n|/n$, we see that the probability measure $\mu_n^{{\bm J},{\bm h}}$  in \eqref{eq:potts_general} is a {generalized version} of the standard Potts measure $\mu_n^{\be, B}$. Throughout most of the article, we will fix a choice of $\{A_n\}_{n \in \N}$, ${\bm J}$, and ${\bm h}$. Therefore, to lighten the notation we will often write $\mu_n(\cdot)$ instead of $\mu_n^{{\bm J}, {\bm h}}(\cdot)$. 

\noindent
Now similarly as before we define the log partition function
\[
\Phi_n({\bm J}, {\bm h}):= \log Z_n({\bm J},{\bm h}).
\]
Arguing same as before we also obtain that 
\beq\label{eq:var_formula}
\Phi_n({\bm J},{\bm h})=\sup_{\gq\in \mathcal{P}([q]^n)} \Big\{\sum_{{\bm y}\in [q]^n}\gq({\bm y}){\bm H}_{n}^{{\bm J}, {\bm h}}({\bm y})-\sum_{{\bm y}\in [q]^n}\gq({\bm y})\log \gq({\bm y})\Big\},
\eeq	
and
\begin{align}\label{eq:var_mean_field_prediction}
	\Phi_n({\bm J},{\bm h})
\ge	\sup_{\gq\in\cP([q])^n}{\mathsf M}_n^{{\bm J},{\bm h}}(\gq),
\end{align}
where 
\begin{align}\label{eq:def_mean_field}
{\mathsf M}_n^{{\bm J},{\bm h}}(\gq):=\Big\{\frac{1}{2}\sum_{i,j=1}^nA_n(i,j)\sum_{r,s=1}^q \gq_i(r)\gq_j(s)J_{rs}+\sum_{i=1}^n\sum_{r=1}^q h_r\gq_i(r)-\sum_{i=1}^n \sum_{r=1}^q\gq_i(r)\log \gq_i(r)\Big\}.
\end{align}

\medskip

\noindent
In Theorem \ref{thm:main} below we show that under a fairly general condition 	\eqref{eq:var_mean_field_prediction} is actually tight as $n \ra \infty$. Before going to the statement of Theorem \ref{thm:main}, for convenience of writing, first let us introduce the following notation:
{\begin{defn}\label{defn:order}
	Let $a_n$ and $b_n$ be two non-negative sequences of real numbers. We write $a_n = o(b_n)$ if $\lim_{n\ra \infty} \frac{a_n}{b_n}=0$, whereas $a_n=O(b_n)$ implies $\limsup_{n\ra \infty} \frac{a_n}{b_n}<\infty$. Note that $a_n=O(b_n)$ includes the possibility of $a_n=o(b_n)$. Next we use the notation $a_n=\Theta(b_n)$, if $a_n=O(b_n)$ and $b_n=O(a_n)$.
\end{defn} }

\noindent
{Note that for both Ising and Potts model we must assume some conditions on $A_n$ to ensure that the resulting log partition is $O(n)$, or equivalently the limiting log partition function to be non-trivial. In this paper we work with the following condition:
\beq\label{eq:model_assumption2}
\sup_{{\bm x} \in [0,1]^n}\sum_{i\in [n]}\left|\sum_{j\in [n]} A_n(i,j) x_j\right| =O(n).
\eeq
Now let us denote $\norm{{\bm J}}_\infty:=\max_{r,s \in [q]}|J_{r,s}|$ and $\norm{{\bm h}}_\infty:=\max_{r \in [q]} |h_r|$. Since
\begin{align*}
|{\bm H}_n^{{\bm J},{\bm h}}({\bm y})|\le& \frac{\norm{\bm J}_\infty}{2}\sum_{i\in [n],r,s\in [q]}\Big|\sum_{j\in [n]}A_n(i,j)\delta(y_j,s)\Big|+ \norm{{\bm h}}_\infty \sum_{i \in [n]} \sum_{r \in [q]} \delta(y_i,r)\\
\le& \frac{\norm{\bm J}_\infty}{2}\sum_{r,s\in [q]}\sup_{{\bm x}\in [0,1]^n}\sum_{i\in [n]}\Big|\sum_{j\in [n]}A_n(i,j)x_j\Big| + n \norm{{\bm h}}_\infty,
\end{align*}
it follows by \eqref{eq:model_assumption2} that $|\sup_{{\bm y}\in [q]^n}H_n^{{\bm J},{\bm h}}({\bm y})|=O(n)$, which implies $\Phi_n({\bm J},{\bm h})=O(n)$ as well.


\noindent
When all entries of $A_n$ have the same sign, condition \eqref{eq:model_assumption2} is equivalent to $$\norm{A_n}_1:=\sum_{i,j\in [n]}|A_n(i,j)|=O(n).$$}
If \eqref{eq:model_assumption2} does not hold then there exists ${\bm J},{\bm h}$ such that the resulting log partition function $\Phi_n({\bm J},{\bm h})$ scales super linearly. For example, if all entries of $A_n$ are positive, ${\bm J}=\be {\bm I}_q$, then for any $\beta>0$ an application of the mean-field lower bound gives
$\lim_{n\rightarrow\infty}\frac{1}{n}\Phi_n(\beta,B)=+\infty$, thus proving that \eqref{eq:model_assumption2} is necessary for the log partition function to be $O(n)$ in general.
If $A_n$  has both positive and negative entries, \eqref{eq:model_assumption2} continues to hold for many well-known models with both positive and negative entries, such as the Sherrington-Kirkpatrick model and Hopfield model (see Section \ref{sub:examples}). 

\bigskip

\noindent
Of course we do not expect the mean-field approximation to hold for all matrices $A_n$ satisfying \eqref{eq:model_assumption2}. For example, it is known that the mean-field approximation is not correct for the {Sherrington-Kirkpatrick} model \cite{Tal}, or Ising models on sparse graphs \cite{dm}. With this in mind we introduce the following definition.

\begin{defn}\label{assu:mf}
Suppose $A_n$ is a sequence of symmetric $n\times n$ matrices safisying \eqref{eq:model_assumption2}. We say that $A_n$ satisfies the {\em mean-field} assumption if $\tr(A_n^2)=o(n)$.
\end{defn}

\medskip

\noindent
Now we are ready to state our first result.
{
\begin{thm}\label{thm:main}
 If  $A_n$ satisfies the mean-field assumption, then  
$$\lim\limits_{n\rightarrow\infty}\frac{1}{n}\left[\Phi_n({\bm J},{\bm h})-\sup_{\gq\in \mathcal{P}([q])^n}{\mathsf M}_n^{{\bm J}, {\bm h}}(\gq)\right]=0.$$
\end{thm}
}
Theorem \ref{thm:main} essentially says that if $A_n$ is a sequence of matrices which satisfies the mean-field assumption then the mean-field approximation gives the right answer for the log partition function upto an error which is $o(n)$.

As an application of Theorem \ref{thm:main}, one immediately obtains the following corollary. This corollary will be used in all of our applications involving graphs.

\begin{cor}\label{cor:main}
Suppose $\Graph_n$ is a sequence of simple graphs, and $A_n$ is the adjacency matrix of $\Graph_n:=([n], E_n)$ multiplied by $n/(2|E_n|)$, where $|E_n|$ is the number of edges. Then the conclusion of Theorem \ref{thm:main} holds if $n=o(|E_n|)$.
\end{cor}

\begin{proof}
Since $$\sup_{{\bm x}\in [0,1]}\sum_{i\in [n]}\Big|\sum_{j\in [n]}A_n(i,j)x_j\Big|=\sum_{i,j\in [n]}A_n(i,j)=n,$$
\eqref{eq:model_assumption2} holds. Also we have
$$\frac{1}{n}\sum_{i,j\in [n]}A_n(i,j)^2=\frac{n}{2|E_n|^2}|E_n|=O\left(\frac{n}{|E_n|}\right)=o(1),$$
and so $A_n$ satisfies the mean-field assumption. The conclusion then  follows by Theorem \ref{thm:main}.
\end{proof}
Below we consider few different choices of $A_n$, and verify for which of those the mean-field assumption is satisfied.

\noindent
\subsection{Examples}\label{sub:examples}

 This is broadly divided into two categories.

\medskip

\subsubsection{\bf Matrices $A_n$ which are scaled adjancency graphs}\label{eg}
\begin{enumerate}[(a)]
\item
Let $\Graph_n$ be any sequence of simple dense labeled graphs on $n$ vertices, i.e.~it has $\Theta(n^2)$ edges. Let $A_n$ be adjacency matrix of $\Graph_n$ scaled by $n$, i.e.  $A_n(i,j):=\frac{1}{n}1_{(i,j) \in E_n}$. Since this scaling is equivalent to the scaling proposed in Corollary \ref{cor:main}, 
it suffices to check that $n=o(|E_n|)$. But this is immediate as $|E_n|=\Theta(n^2)$. 
 
 \smallskip

\item 
Let $\Graph_n$ be a  $d_n$ regular graph, and $A_n(i,j):=\frac{1}{d_n}1_{(i,j) \in E_n}$.  In this case again the scaling is the same one as that of Corollary \ref{cor:main}, and so it suffices to check that $n=o(|E_n|).$ Since $2|E_n|=nd_n$, Corollary \ref{cor:main} holds iff $d_n\rightarrow\infty$.

\smallskip

\item
Let $\Graph_n$ be an Erd\H{o}s-R\'{e}nyi random graph with parameter $p_n$.  Setting $A_n(i,j):=\frac{1}{np_n}1_{(i,j) \in E_n}$ it again suffices to check by Corollary \ref{cor:main} that $n=o(|E_n|)$, in probability. Since $|E_n|$ has $\dBin\left({n \choose 2}, p_n\right)$ distribution,
$$\frac{2|E_n|}{n^2p_n}{\rightarrow}1, \text{ in probability},$$
as soon as $n^2p_n\rightarrow\infty$, 
the mean-field assumption holds in probability iff $np_n\rightarrow \infty$. In particular the mean-field condition does not hold if $p_n=\frac{\lambda}{n}$ for some $\lambda<\infty$.

\smallskip

\item
Let $\Graph^{(d)}_n$ be the $[-n^{1/d},n^{1/d}]^d$ box of the $d$-dimensional integer lattice $\Z_d$.
Physicists have long been interested in studying Ising and Potts models on lattices (see \cite{N1,wu}, and the references therein). For any finite $d$, setting $A^{(d)}_n(i,j):=\frac{1}{d}1\{(i,j)\in E_n\}$ we note that $\f{1}{n}\tr((A_n^{(d)})^2)= O(\frac{1}{d})$, and thus the sequence does not satisfy the mean-field assumption. So our results are not applicable on $\Z^d$ for finite $d$. However, if we allow $d$ to go to infinity (at any rate) along with $n$, then Corollary \ref{cor:main} is applicable. One can check that this also implies that if we let  $d \ra \infty$ after letting $n \ra \infty$, the same conclusion continues to hold. Behavior of limiting log-partition function for the Potts model on $\Z^d$ for large $d$ has been studied in \cite{BC,CP}. We recover their results as an application of Corollary \ref{cor:main}

\end{enumerate}

\smallskip

\subsubsection{\bf Matrices with both positive and negative entries}

A general sufficient condition for  \eqref{eq:model_assumption2} to hold is $\norm{A_n}:=\sup_{{\bm x}: \norm{\bm x}_2=1}\norm{A_n {\bm x}}_2=O(1)$. To see this note that an application of Cauchy-Schwarz inequality gives
\begin{align*}\sup_{{\bm x}\in [0,1]^n}\sum_{i\in [n]}\Big|\sum_{j\in [n]}A_n(i,j)x_j\Big|\le \sqrt{n}\sup_{{\bm x}\in [0,1]^n}\norm{A_n {\bm x}}_2\le \sqrt{n}\norm{A_n}\sup_{{\bm x}\in [0,1]^n}\norm{\bm x}_2=O(n).
\end{align*}
\begin{enumerate}[(a)]
\item
Let $A_n$ be a symmetric matrix with $0$ on the diagonal, and $A_n(i,j)=\frac{1}{\sqrt{n}}Z(i,j)$ with $$\{Z(i,j)\}_{1\le i<j\le \infty}\stackrel{i.i.d.}{\sim}N(0,1).$$ This is the celebrated {Sherrington-Kirkpatrick} model of statistical physics introduced in \cite{SK}. Since $\norm{A_n}=O(1)$, in probability, in this case (see \cite[Theorem 2.12]{Bai}), \eqref{eq:model_assumption2} holds. However $A_n$ does not satisfy the mean-field assumption, as
$$\frac{1}{n}\sum_{i,j\in [n]}A_n(i,j)^2=\frac{1}{n^2}\sum_{i,j\in [n]}Z(i,j)^2{\rightarrow 1}, \text{ in probability}.$$
This is expected, as the log partition function in this case is given by the Parisi formula, and not by the mean-field approximation.

\item
Let $\eta$ be an $n\times m$ matrix of i.i.d.~random variables with $\P(\eta_{ik}=\pm 1)=\frac{1}{2}$, and let $$A_n(i,j)=\frac{1}{n}\sum_{k\in [m]}\eta_{ik}\eta_{jk}.$$
This is the Hopfield model of neural networks, first introduced in \cite{Hop}. In this case also one has $\norm{A_n}=O(1)$, in probability, when $m =\Theta(n)$ (see \cite[Section 2.2.2]{Bai}), and therefore \eqref{eq:model_assumption2} holds. Proceeding to check the mean-field condition one has
$$\frac{1}{n}\E\sum_{i,j\in [n]}A_n(i,j)^2=\frac{1}{n^3}\sum_{i,j\in [n]k,l\in [m]}[\delta(i,j)+\delta(k,l)-\delta(i,j)\delta(k,l)]=\frac{nm^2+n^2m-mn}{n^3},$$
and so the mean-field condition does not hold for $m=\Theta(n)$.

\end{enumerate}

\bigskip

\subsection{Proof technique}\label{sec:proof technique}

Establishing the conclusion of  Theorem \ref{thm:main} for graphs whose adjacency matrix has a single dominant eigenvalue is much easier, since in that case the behavior of the log partition function is governed by that eigenvalue. This is indeed the case for Erd\H{o}s-R\'enyi random graphs on $n$ vertices with parameter $p_n$ such that $np_n\gg \log n$. For example, in this regime the largest eigenvalue equals $np_n(1+o(1))$ (see \cite[Section 1]{ks}), whereas the second largest eigenvalue is $o(np_n)$ (see \cite[Theorem 1.1]{FO}), providing a spectral gap. Similarly for random $d_n$-regular graphs on $n$ vertices, one also has a spectral gap, as long as $d_n \ge (\log n)^\gamma$ for some $\gamma$ positive (see \cite{CL,CLV}). More generally, any {\em expander graph} has a spectral gap, and therefore for such graphs one can show that the mean-field approximation is asymptotically tight.  However, there are many graphs which are not expanders, such as the $d$-dimensional hypercube $\{0,1\}^d$ with $d\rightarrow\infty$. In this case the number of vertices in the graph is $n=2^d$, and it is well known  that the set of eigenvalues are $\{d-2i,0\le i\le d\}$ with multiplicity of $d-2i$ being ${d\choose i}$. Thus the two largest eigenvalues are $d$ and $d-2$ whose ratio converges to $1$ as $d$ becomes large, and consequently there is no dominant eigenvalue.

Even though there is no spectral gap in the hypercube, it is still the case that the number of {\em big} eigenvalues is small. For example, the largest eigenvalue is $d$, and the proportion of eigenvalues that lie outside the interval $[-d\delta,d\delta]$, for any  $\delta>0$, equals
$$\frac{1}{n}\sum_{i=1}^n1\{|d-2i|>d\delta\}=\P\left(\Big|\frac{1}{d}\sum_{i\in [d]}B_i-\frac{1}{2}\Big|>\delta\right),$$
where $\{B_i\}_{i\in [d]}$ are i.i.d.~Bernoulli random variables with $\P(B_i=0)=\P(B_i=1)=.5$. By weak law of large {numbers} the RHS above is $o(1)$, as $d\ra \infty$, and so the proportion of eigenvalues which are comparable to the leading eigenvalue is $o(1)$. Our proof makes this precise proving Theorem \ref{thm:main}  which covers not just the hypercube, but any sequence of graphs $\Graph_n$ satisfying $n=o(|E_n|)$ (see Corollary \ref{cor:main}). In fact the main condition of Theorem \ref{thm:main}. i.e.~the condition $\tr(A_n^2)=o(n)$, can be rewritten as $$\frac{1}{n}\sum_{i=1}^n\lambda_i(A_n)^2=o(1),$$
which says that the (properly scaled) empirical eigenvalue distribution converges to $0$ in $L^2$. And of course, as already pointed out that the mean-field  approximation does not hold in general when $|E_n|=\Theta(n)$, thus demonstrating that the conditions of Theorem \ref{thm:main}, and Corollary \ref{cor:main} are tight.

The main tool in the proof of Theorem \ref{thm:main} is a modified version of \cite[Theorem 1.5]{chatterjee_dembo}. For readers not familiar with \cite{chatterjee_dembo}, we informally describe the theorem and the ideas behind the proof of \cite[Theorem 1.5]{chatterjee_dembo}. Before proceeding, we define the notion of a {\em net} of a set.
\begin{defn}\label{defn:net}
For any $S\subset \R^n$ and $\varepsilon>0$, a set $\widetilde{S}\subset \R^n$ is said to be a $\varepsilon$ net of $S$, if given $s\in S$ there exists (at least one) $\widetilde{s}\in \widetilde{S}$ such that $\|s-\widetilde{s}\|_2\le \varepsilon$.
\end{defn}

The theorem assumes that $f:[0,1]^n\mapsto \R$ is a smooth function such that the set $\{\nabla f({\bm u}):{\bm u}\in \{0,1\}^n\}$ has an {\em $\sqrt{n}\vep$} net $\cD_n(\vep)$ with $\log |\cD_n(\vep)|=o(n)$, and concludes that
$$\log \sum_{{\bm u}\in \{0,1\}^n} e^{ f({\bm u})}= \sup_{{\bm u}\in [0,1]^n}\{f({\bm u})-I_n({\bm u})\}+o(n),$$
where $I_n({\bm u}):=\sum_{i=1}^n u_i\log u_i+(1-u_i)\log(1-u_i)$ is the binary entropy function, and ${\bm u}:=(u_1,\ldots,u_n)$. 

For the proof, they introduce a measure $\nu_n(\cdot)$  on $\{0,1\}^n$ given by $\nu_n({\bm u}) \propto \exp(f({\bm u}))$ for  ${\bm u}:=(u_1,u_2,\ldots,u_n) \in \{0,1\}^n$. First it is argued that $f({\bm u})$ and $f(\wh{\bm u})$ are close on a set with {\em high probability} under $\nu_n(\cdot)$, say $\cA_n$ (see \cite[Lemma 3.1]{chatterjee_dembo}). Here $\hat{u}_i$ is conditional expectation of $u_i$, conditioned on everything else. Therefore $\sum_{{\bm u}\in \{0,1\}^n} \exp(f({\bm u}))$ can be well approximated by $\sum_{{\bm u} \in \cA_n} \exp(f(\wh{\bm u}))$. Turning to evaluate the latter summation, it is further noted that $g({\bm u}, \wh{\bm u})$, and $I_n(\wh{\bm u})$ are also close on $\cA_n$ (see  \cite[Lemma 3.2]{chatterjee_dembo}), where for ${\bm u} \in [0,1]^n$, and ${\bm w} \in (0,1)^n$,
\[
g({\bm u}, {\bm w}):= \sum_{i \in [n]} u_i \log w_i + (1-u_i) \log (1-w_i), \text{ and } I_n({\bm w}):= g({\bm w}, {\bm w}).
\]
Therefore one only needs to control $\sum_{{\bm u} \in \cA_n}  \exp(f(\wh{\bm u})+ g({\bm u}, \wh{\bm u}) - I_n(\wh{\bm u}))$. To control the above, the summation over $\cA_n$ is broken into smaller sets where each sum is over only those ${\bm u}$ for which $\wh{\bm u} \approx {\bm p}$, for some ${\bm p} \in [0,1]^n$. Next instead of summing over all choices of ${\bm p} \in [0,1]^n$, the sum is restricted on the {$\sqrt{n}\vep$-net of the image of the map ${\bm u}\mapsto \wh{\bm u}$, using the set $\cD_n(\vep)$. Thus one obtains
\begin{align}
\log \sum_{{\bm u} \in \cA_n}  \exp(f(\wh{\bm u})+ g({\bm u}, \wh{\bm u}) - I_n(\wh{\bm u})) \approx \log \sum_{{\bm p} \in \cD_n(\vep)} \sum_{{\bm u}: \wh{\bm u} \approx {\bm p}} \exp(f({\bm p})+ g({\bm u}, {\bm p}) - I_n({\bm p})). 
\end{align}
Finally noting that
\[\sum_{{\bm u} \in \{0,1\}^n} e^{g({\bm u}, {\bm p})}=1,
\]
the proof follows as the size of $\cD_n(\vep)$ is sub-exponential.

\medskip

In our proof we follow the same scheme. However, there are several challenges that we had to overcome to apply this idea in our set-up. First, we need to find a net $\cD_n(\vep)$ with appropriate properties. In our set-up, we need to find a $\sqrt{n}\varepsilon$-net $\cD_n(\varepsilon)$  of the set $\{A_n{\bm v}:{\bm v}\in \{0,1\}^n\}$. Since we have very limited assumptions on the structure of $A_n$, obtaining a $\sqrt{n}\vep$-net is not straightforward.  The main difficulty comes from the fact that the eigenvalues of  $A_n$ can be unbounded. To overcome this, we split the range of the eigenvalues into its {\em level sets}, and then we choose nets of varying size across each of the level sets (for more details see proof of Lemma \ref{lem:mean_field}). 

 Equipped with Lemma \ref{lem:mean_field}, a direct application of \cite[Theorem 1.5]{chatterjee_dembo} proves Theorem \ref{thm:main} for graphs $\Graph_n$ such that \begin{align}\label{eq:chat_dem}
	\limsup_{n\rightarrow\infty}n\sum_{i\in [n]}\left(\frac{d_i(\Graph_n)}{\sum_{j\in [n]}d_j(\Graph_n)}\right)^2<\infty,
	\end{align}
where $\{d_1(\Graph_n),\cdots,d_n(\Graph_n)\}$ are the degrees of $\Graph_n$. The hypercube does satisfy this condition, as does any regular graph. 
There are many graphs in literature such that 
 $n=o(|E_n|)$,  but \eqref{eq:chat_dem} does not hold. For example, let $\Graph_n$ denote the complete bipartitle graph $K_{a_n,n-a_n}$, where $a_n$ is a  sequence of natural numbers going to $\infty$ such that $a_n=o(n)$. In this case the LHS of \eqref{eq:chat_dem} equals $$\frac{n[a_n(n-a_n)^2+(n-a_n)a_n^2]}{4a_n^2(n-a_n)^2}=O\Big(\frac{n}{a_n}\Big),$$
which is not $O(1)$, as $a_n=o(n)$. Since $|E_n|=a_n(n-a_n)$ with $a_n\rightarrow\infty$, Corollary \ref{cor:main} is still applicable for $K_{a_n, n-a_n}$ but \cite[Theorem 1.5]{chatterjee_dembo} does not apply.

To remove the requirement of \eqref{eq:chat_dem} we modify the proofs of \cite[Lemma 3.1]{chatterjee_dembo}, and \cite[Lemma 3.2]{chatterjee_dembo}. In the proof of these two lemmas, at many places, supremum norm bound is used for several functions. The condition \eqref{eq:chat_dem} arises because of that. Instead, we carefully use the assumption \eqref{eq:model_assumption2}, and the fact that the hamiltonian in our set-up is a quadratic function.  This part of the proof has been inspired from \cite{Chatterjee}.  

\medskip

\noindent
In Section \ref{sec:appl} we provide several applications of Theorem \ref{thm:main}. One of which is the computation of the limit for asymptotically regular graphs. To be more precise, we call a sequence of graphs to be asymptotically regular if the empirical distribution of the row sums of the properly scaled adjacency matrix converges to $\delta_1$, and if its mean also converges to one. Using a truncation argument we derive the desired result. We also find the limit for bi-regular bipartite graphs, for which we carefully analyze the solutions of some fixed point equations. Lastly, we identify the limit for a sequence of simple graphs converging in cut metric. This follows from a straightforward analysis upon using Theorem \ref{thm:main}.

\subsection{Outline}
The outline of the rest of the paper is as follows. As applications of Theorem \ref{thm:main}, in Section \ref{sec:appl} we derive {the} asymptotics of the log partition function for ferromagnetic Potts models on asymptotically regular graphs, that of Ising models (both ferromagnetic and anti-ferromagnetic) on bi-regular bipartite graphs, and that of Potts model on a sequence of simple graphs converging in cut metric in the $L_p$ sense.  Section \ref{sec:Pot} carries out the proof of Theorem \ref{thm:main} using three auxiliary lemmas, whose proofs are deferred to Section \ref{sec:Pol}. Finally in Section \ref{sec:Poa} we prove the results appearing in Section \ref{sec:appl}.

\noindent
{\bf Acknowledgements.}	We thank Andrea Montanari for suggesting to look at the Ising measure on hypercube, Sourav Chatterjee for pointing out the reference \cite{chatterjee_dembo}, and Amir Dembo for helpful comments on earlier version of the manuscript. We also thank Sourav Chatterjee, Amir Dembo, and Andrea Montanari for many helpful discussions. We further thank Marek Biskup and Aernout Van Enter for pointing out the references \cite{BC} and \cite{CP} respectively. We are grateful to two anonymous referees for their detailed comments and suggestions which have improved the quality of this paper.

\section{Applications of theorem \ref{thm:main}}\label{sec:appl}
\noindent

\subsection{Asymptotically regular graphs}
In Theorem \ref{thm:main} we saw that the mean-field prediction is asymptotically correct when $A_n$ satisfies the mean-field condition. However, computing the supremum of ${\sf M}_n^{{\bm J}, {\bm h}}(\gq)$ may often be very hard for general matrices $A_n$. Restricting ourselves to the case ${\bm J}=\be I_q$ for $\beta>0$, in Theorem \ref{thm:mean_field} below we show that when the matrices $A_n$ are ``asymptotically regular'' one can write the $n$-dimensional supremum as a one-dimensional supremum, and thereby providing more tractable form of the limit. In particular, setting $h_r=B\delta(r,1)$, for asymptotcally regular graphs the limit is same as the one obtained for a Curie-Weiss Potts model.

\begin{thm}\label{thm:mean_field}
			\begin{enumerate}[(a)]
			\item
				Let $A_n$ satisfies the mean-field assumption, and each entry of $A_n$ is non-negative. Also let ${\bm J}=\beta I_q$, for some $\beta\ge 0$. Set $\cR_n(i):=\sum_{j=1}^nA_n(i,j)$. If 
		
				 \begin{align}\label{eq:upper_mean_field}
				 \lim_{n\rightarrow\infty}\frac{1}{n}\sum_{i =1}^n \delta_{\cR_n(i)}{\rightarrow}\delta_1, \text{ in distribution},
				 \end{align}
				 and
				 \begin{align}
				 \label{eq:lower_mean_field}
				 \lim_{n\rightarrow\infty}\frac{1}{n}\sum_{i=1}^n\cR_n(i)= 1,
				\end{align}
				then 
				\begin{align}\label{eq:strong_meanfield_conclusion}
				\lim_{n\rightarrow\infty}\frac{1}{n}\Phi_n({\bm J},{\bm h}) = \sup_{\gq\in \cP([q])}\Big[\frac{\beta}{2}\sum_{r=1}^q\gq(r)^2-\sum_{r=1}^q \gq(r)\log \gq(r)+\sum_{r=1}^q h_r\gq(r)\Big].
				\end{align}
			\item
			In particular, the conclusion of part (a) applies  in the following two cases:
			
			\begin{enumerate}[(i)]
			\item
			$\Graph_n$ is a sequence of $d_n$ regular graphs with $d_n\rightarrow\infty$, and $A_n=\frac{1}{d_n}1_{(i,j) \in E_n}$.
			
			\item
			$\Graph_n$ is an Erd\H{o}s-R\'{e}nyi random graph with parameter  $p_n$ such that $np_n\rightarrow\infty$, and $A_n=\frac{1}{np_n}1_{(i,j) \in E_n}$.
			\end{enumerate}
		\end{enumerate}
			\end{thm}

\medskip

\noindent
{As an application of the above theorem, the following theorem derives the large deviation for the empirical measure $L_n$ on $\cP([q])$ defined by
		$$L_n(r):=\frac{1}{n}{\sum_{i\in [n]}}\delta(y_i,r).$$ 
		
		Below we recall a  few definitions of large deviation theory  which are necessary for our paper.

		\begin{defn} 
		Let $(\cX,\cB)$ be a measure space equipped with a topology such that every open set is in $\cB$.
A function $I:\cX\mapsto [0,\infty]$ is said to be a rate function if it is lower semi continuous, i.e. for every $\alpha<\infty$ the set $\{x\in\cX:I(x)\le \alpha\}$ is closed. The function $I$ is said to be a good rate function, if further the set  $\{x\in\cX:I(x)\le \alpha\}$ is compact as well. In particular if $\cX$ is compact, any rate function is a good rate function.

\smallskip

\noindent
A sequence of probability measures $\P_n$ on $(\cX,\cB)$ is said to satisfy a large deviation on $\cX$ with respect to a good rate function $I(\cdot)$, at speed $n$, if for every closed set $F$, and open set $U$, we have 
\begin{align*}
\limsup_{n\rightarrow\infty}\frac{1}{n}\log \P_n(F)\le -\inf_{x\in F}I(x),\\
\liminf_{n\rightarrow\infty}\frac{1}{n}\log \P_n(U)\ge -\inf_{x\in U}I(x).
\end{align*}
\end{defn}		
		
\medskip

\noindent		
The large deviation reduces the concentration of measure problem to an optimization problem involving the rate function. Next we introduce a few notations which will be needed while solving this optimization problem.
		\begin{defn}\label{lem:tanh0}
		For $\beta>0,B\ne 0$ let $m_{\beta,B}$ denote the unique solution of $m=\tanh(\beta m+B)$ with the same sign as that of $B$. For $\beta>1,B=0$ let $m_{\beta,0}$ denote the unique positive root of the equation $m=\tanh(\beta m)$. The assertions  about the roots of the equation $m=\tanh(\beta m+B)$ can be found in \cite[Section 1.1.3]{dm}.
%
%
%
%
		\end{defn}

		\begin{thm}\label{thm:ldp} (a) In the setting of Theorem \ref{thm:mean_field}, the sequence of empirical measures $L_n$ satisfies a large deviation principle on $\cP([q])$ with speed $n$ with respect to Euclidean topology, with the good rate function $\widetilde{I}_{\beta,{\bm h}}(\mu):=I_{\beta,{\bm h}}(\mu) -\min_{\mu\in \cP([q])}I_{\beta,{\bm h}}(\mu)$, where $$I_{\beta,{\bm h}}(\mu):=\sum_{r\in [q]}\Big(\mu_r\log \mu_r-\f{\beta\mu_r^2}{2}-h_r\mu_r\Big).$$ 
	Consequently letting $K_{\beta,{\bm h}}:=\arg\min_{\mu\in \cP([q])}I_{\beta,{\bm h}}(\mu)$, for any $\delta>0$  we have
\beq\label{eq:conc}
\limsup_{n\rightarrow\infty} \frac{1}{n}\log\mu_n(\min_{\mu\in K_{\beta,{\bm h}}}\norm{L_n-\mu}_\infty\ge \delta)<0.
\eeq
		
\noindent
(b) 
%
%
%
%
%
%
%
Suppose we are in the setting of Theorem \ref{thm:mean_field} with $q=2$ (which corresponds to Ising model).

\begin{enumerate}[(i)]
\item
If $h_1-h_2=0$ then

\begin{itemize}
\item
For $\beta\le  2$, for any $\delta>0$ there exists $\varepsilon=\varepsilon(\beta,\delta)$ such that for all large $n$ we have
$$\mu_n\left(\frac{1}{n}\sum_{i\in [n]}\Big\{\delta(y_i,1)-\delta(y_i,2)\Big\}\in [-\delta,\delta]\right)\ge 1-e^{-n\varepsilon}.$$

\item
For $\beta> 2$, for any $\delta>0$ there exists $\varepsilon=\varepsilon(\beta,\delta)$ such that for all large $n$ we have
\begin{align*}
&\mu_n\left(\frac{1}{n}\sum_{i\in [n]}\Big\{\delta(y_i,1)-\delta(y_i,2)\Big\}\in  [m_{\beta/2,0}-\delta,m_{\beta/2,0}+ \delta]\right)\ge \frac{1}{2}-e^{-n\varepsilon},\\
&\mu_n\left(\frac{1}{n}\sum_{i\in [n]}\Big\{\delta(y_i,1)-\delta(y_i,2)\Big\}\in [-m_{\beta/2,0}-\delta,-m_{\beta/2,0}+ \delta]\right)\ge \frac{1}{2}-e^{-n\varepsilon},
\end{align*}
where $m_{\beta,0}$ is as in Definition \ref{lem:tanh0}.
\end{itemize}

\item
If $h_1-h_2=B\ne 0$, for any $\delta>0$ there exists $\varepsilon=\varepsilon(\beta,B,\delta)$ such that for all large $n$ we have
$$\mu_n\left(\frac{1}{n}\sum_{i\in [n]}\Big\{\delta(y_i,1)-\delta(y_i,2)\Big\}\in [m_{\beta/2,B/2}-\delta,m_{\beta/2,B/2}+\delta] \right)\ge 1- e^{-n\varepsilon},$$
where $m_{\beta,B}$ is as in Definition \ref{lem:tanh0}.

\end{enumerate}

\end{thm}
\begin{remark}
Theorem \ref{thm:ldp}(b) gives concentration results for $\frac{1}{n}\sum_{i\in [n]}\{\delta(y_i,1)-\delta(y_i,2)\}$, for the Ising model, i.e.~for the Potts model of \eqref{eq:basic_potts} for $q=2$. If the Ising model is formulated in such a way that the spins take values in $\{-1,1\}$, then one can easily see that the results of Theorem \ref{thm:ldp}(b) are equivalent to the exponential concentration of average spin configuration in that set-up. 
	This gives a complete picture for the ferromagnetic Ising model $\mu_n^{\beta {\bm I}_2,{\bm h}}$ for all choices of the vector ${\bm h}$, for asymptotically regular graphs. The optimization of $I_{\beta,{\bm h}}$ for general $q$ for some specific choices of ${\bm h}$ is well known in the literature (see \cite{BGRW,cet,em,ew,ew_mle}). Using these results similar concentration results can be derived for the Potts model on asymptotically regular graphs, for those choices of ${\bm h}$. We omit the details.

	\end{remark}

%

		\subsection{Ising model on bipartite graphs}	
		
		This section focuses on the Ising model ($q=2$) on bipartite graphs. 
		\begin{defn}
		Let $\Graph_{(a,b),(c,d)}$  denote a bi-regular bipartite graph on $a+b$ labeled vertices, such that the two partite sets have sizes $a$ and $b$, and the common degree of vertices in those two partite sets  are $c$ and $d$ respectively. Thus we must have  $ac=bd$, which equals the number of edges.

\vskip.1cm
\noindent
In particular $\Graph_{(a,b),(b,a)}$ denotes the complete bipartite graph with the two partite sets having sizes $a$ and $b$.
		
		\end{defn}
	
\medskip		
		
		\begin{defn}\label{lem:tanh}
		For any $p\in (0,1)$ and $\beta\in \R$ set $\eta_{\beta,p}(\gs):=\tanh(\beta (1-p)\tanh(\beta p \gs))$. By elementary calculus it follows that 
	
\noindent	
(a) For $\be^2 p(1-p) \le 1$ the equation $\gs=\eta_{\beta,p}(\gs)$ has the unique root $0$.

\noindent
(b) For $\be^2 p(1-p) >1$ the equation $\gs=\eta_{\beta,p}(\gs)$ has a unique positive root, denoted hereafter by $\gs_{\be,p}$. Thus the aforementioned equation has three roots, namely $0,\gs_{\beta,p}$, and $-\gs_{\beta,p}$. Applying implicit function theorem, we also note that the function $(\be, p)\mapsto \gs_{\beta,p}$ is a continuously differentiable in the open set $\{(\be,p):p(1-p)\be^2>1\}$.
			\end{defn}

%
		\begin{thm}\label{thm:bipartite}
		
		Let $\Graph_{(a_n,n-a_n),(c_n,d_n)}$ be a sequence of bipartite graphs on $n$ labeled vertices, such that
 \begin{align}\label{eq:alpha}
 \lim_{n\rightarrow\infty}\frac{a_n}{n}=p\in (0,1),
 \end{align} and  $c_n+d_n \ra \infty$, as $n \ra \infty$. Thus for $q=2$, ${ \bm J}=\beta {\bm I}_2$ for some $\beta\in\R$, ${\bm h}={\bm 0}$ in \eqref{eq:potts_general}, setting $A_n$ to be the adjacency matrix of $\Graph_{(a_n,n-a_n),(c_n,d_n)}$ scaled by $c_n+d_n$ we have

\noindent
(a) If $\beta^2 p(1-p) \le 1$, then 
		$$\lim_{n\rightarrow\infty}\Phi_n(\beta,0)=\frac{\beta p(1-p)}{2}+\log 2.$$
(b) If $\beta^2 p(1-p) > 1$, then	
		$$\lim_{n\rightarrow\infty}\Phi_n(\beta,0)=\frac{\beta p(1-p)}{2}+\frac{|\beta| p(1-p)}{2}\gs_{|\beta|,p}\gs_{|\beta|,1-p}+p H(\gs_{|\beta|,p})+(1-p) H(\gs_{|\beta|,1-p}), $$
where  $\gs_{\beta,p}(\cdot)$ is as in Definition \ref{lem:tanh}, and $H(\gs):=-\frac{1+\gs}{2}\log \frac{1+\gs}{2}-\frac{1-\gs}{2}\log \frac{1-\gs}{2}$ for $\gs \in [-1,1]$.

		\end{thm}

		\subsection{Potts model on converging sequence of graphs in cut metric }\label{sec:cut metric}
		\vspace{.2cm}
		
		The theory of dense graph limits was developed by Borgs, Chayes, Lovasz, and coauthors \cite{graph_limits_I, graph_limits_II, lovasz_book}, and has received phenomenal attention over the last few years. 
Recent works of Borgs et al~\cite{sparse1,sparse2} have extended  this theory beyond the regime of dense graphs. 
		One of the results in \cite{sparse2} is the asymptotics of the log partition function $\Phi_n({\bm J},{\bm h})$ of \eqref{eq:potts_general} of a  sequence of graphs converging in the sense of cut metric to functions $W$ that are unbounded. As a byproduct of Theorem \ref{thm:main} we are able to provide a short proof of their result. Before going to the statement of the result, we first need to introduce necessary notations, and concepts. These are taken from \cite{sparse1,sparse2}.
		
		\begin{defn}
			A function $W:[0,1]^2 \mapsto \R$ is called a symmetric function if $W(x,y) = W(y,x)$ for all $x,y \in [0,1]$. Any symmetric measurable function $W:[0,1]^2\mapsto \R$ which is $L^1$ integrable, i.e.~$\norm{W}_1:=\int_{[0,1]^2}|W(x,y)|dxdy<\infty$ is called a graphon.
			
\noindent			
Given a symmetric $n\times n$ matrix $A_n$, define a graphon on $[0,1]^2$ by dividing $[0,1]^2$ into $n^2$ smaller squares each of length $1/n$, and setting
			$W_{A_n}(x,y):=A_n(i,j)$ if $(x,y)$ is in the $(i,j)$-th box, i.e. $\lceil nx\rceil =i, \lceil ny\rceil =j$.

\noindent
The cut norm of a graphon $W$ is given by
			 $$\norm{W}_\square=\Big|\sup_{S,T\subset [0,1]}\int_{S\times T}W(x,y)dxdy\Big|.$$
			  After identifying graphons with cut distance zero, the set of equivalences classes of graphons equipped with the cut metric is a compact metric space. The cut norm is equivalent to the  $L^\infty\mapsto L^1$ operator norm defined by
			 
			 $$\norm{W}_{\infty\mapsto 1}:=\sup_{f,g: \norm{f}_\infty, \norm{g}_\infty \le 1}\Big|\int_{[0,1]^2} W(x,y)f(x)g(x)dxdy\Big|.$$
More precisely, we have
		$\norm{W}_\square\le \norm{W}_{\infty\mapsto 1}\le 4\norm{W}_\square.$
		
		\end{defn}
		
\medskip		
Next we introduce the notion of fractional partition.		
		\begin{defn}
			A $q$ tuple of measurable functions ${\bm \rho}:=(\rho_1,\cdots,\rho_q): [0,1]^q\mapsto [0,1]^q$, such that
			$$\sum_{r\in [q]} \rho_r(x)=1,\forall x\in [0,1],$$ will be called a fractional partition of $[0,1]$ into $q$ classes. The set of fractional partitions of $[0,1]$ into $q$ classes will be denoted by ${\bf FP}_q$.
			\end{defn}
			
\medskip
Now we are ready to state the result about the limiting log partition function for {\em a sequence of graphs converging in cut metric}.				
		\begin{thm}\label{thm:lp}
		
		Let $\Graph_n$ be a  sequence of simple graphs, and let $A_n$ be the adjacency matrix of $\Graph_n$ scaled by $\frac{2|E_n|}{n}$. If $W_{nA_n}$ converges in cut metric to a graphon $W$, then we have
			\begin{align*}\lim_{n\rightarrow\infty}\frac{1}{n}\Phi_n({\bm J},{\bm h})=&\sup_{{\bm \rho}\in {{\bf FP}_q}}F^{{\bm J},{\bm h}}(W,{\bm \rho}),
				\end{align*}
				where
				\begin{align*}
			F^{{\bm J},{\bm h}}(W,{\bm \rho})& :=\frac{1}{2}\sum_{r,s\in [q]} J_{rs}\int_{[0,1]^2}\rho_r(x)\rho_s(y)W(x,y)dxdy\\
			& \qquad \qquad \qquad \qquad \qquad +\sum_{r\in [q]} h_r\int_{[0,1]}\rho_r(x)dx-\int_{[0,1]}\sum_{r\in [q]}\rho_r(x)\log \rho_r(x)dx.
			\end{align*}
			
			\end{thm}

\medskip
Theorem \ref{thm:lp} follows from \cite[Theorem 2.10]{sparse2}, and \cite[Lemma 3.2]{sparse2}. In section \ref{sec:Poa} we give a shorter proof of the same using Corollary \ref{cor:main}.

\section{Proof of theorem \ref{thm:main}}\label{sec:Pot}

We begin with a simple lemma which allows us to assume that the entries of $A_n$ are $o(1)$.

\begin{lem}
Let $A_n$ be a sequence of matrices that satisfies the mean-field assumption. Then there is a sequence of matrices $\widetilde{A}_n$ with $0$ diagonal entries which also satisfies the mean-field assumption such that $\max_{i,j\in [n]}|\widetilde{A}_n(i,j)|=o(1)$, and
\begin{align*}
|\Phi_n({\bm J},{\bm h})-\widetilde{\Phi}_n({\bm J},{\bm h})|=o(n),\quad \sup_{\gq\in \cP([q])^n}|{\sf M}_n^{{\bm J}, {\bm h}}(\gq)-\widetilde{{\sf M}}_n^{{\bm J}, {\bm h}}(\gq)|=o(n),
\end{align*}
where $\widetilde{\Phi}_n({\bm J},{\bm h})$ and $\widetilde{{\sf M}}_n^{{\bm J}, {\bm h}}(\gq)$ {are obtained by replacing $A_n$ with $\widetilde{A}_n$ in the corresponding definitions}.
\end{lem}
\begin{proof}
Since $A_n$ satisfies the mean-field assumption, setting $\varepsilon_n:=n^{-1/2}\sqrt{\tr(A_n^2)}$, we see that $\vep_n\rightarrow 0$. Now defining an $n\times n$ symmetric matrix $\widetilde{A}_n$ by 
$$\widetilde{A}_n(i,i):=0, \quad \widetilde{A}_n(i,j):=A_n(i,j)1_{|A_n(i,j)|\le \varepsilon_n},$$
one immediately has $\max_{i,j\in [n]}{|\widetilde{A}_n(i,j)|}\le \varepsilon_n\rightarrow 0$. Extending the definition of ${\bm H}_n^{{\bm J}, {\bm h}}(\cdot)$ to $\cP([q])^n$ (see Definition \ref{defn:defn} below for more details), and defining $\widetilde{{\bm H}}^{{\bm J},{\bm h}}_n$ analogously one has
\begin{align}\label{eq:entry_small}
\sup_{\gq\in \cP([q])^n}\Big|{\bm H}_n^{{\bm J},{\bm h}}(\gq)-\widetilde{{\bm H}}_n^{{\bm J},{\bm h}}(\gq)\Big|\le &\frac{q\norm{{\bm J}}_\infty}{2}\left(\sum_{i,j\in [n]}|A_n(i,j)|1_{|A_n(i,j)|>\varepsilon_n})+\sum_{i \in [n]}|A_n(i,i)|\right)\notag\\
\le &\frac{q\norm{{\bm J}}_\infty}{2\varepsilon_n}\sum_{i,j\in [n]}A_n(i,j)^2+\frac{q\norm{{\bm J}}_\infty}{2}\sqrt{n \sum_{i \in [n]}A_n(i,i)^2}\notag\\
\le & \frac{nq\norm{{\bm J}}_\infty\varepsilon_n}{2} +\frac{q\norm{{\bm J}}_\infty}{2}\sqrt{n \tr(A_n^2)} =o(n),
\end{align}
which immediately implies $\sup_{\gq\in \cP([q])^n}|{\sf M}_n^{{\bm J}, {\bm h}}(\gq)-\widetilde{{\sf M}}_n^{{\bm J}, {\bm h}}(\gq)|=o(n)$. Also we have
\begin{align*}
\left|\Phi_n({\bm J},{\bm h})-\widetilde{\Phi}_n({\bm J},{\bm h})\right|=\left|\log\frac{\sum_{{\bm y}\in [q]^n}e^{{\bm H}_n^{{\bm J},{\bm h}}({\bm y})}}{\sum_{{\bm y}\in [q]^n}e^{\widetilde{{\bm H}}_n^{{\bm J},{\bm h}}({\bm y})}}\right|\le& \sup_{{\bm y}\in [q]^n} |{\bm H}_n^{{\bm J},{\bm h}}({\bm y})-\widetilde{{\bm H}}_n^{{\bm J},{\bm h}}({\bm y})|\\
\le&\sup_{{\gq }\in \cP([q]^n)} |{\bm H}_n^{{\bm J},{\bm h}}(\gq)-\widetilde{{\bm H}}_n^{{\bm J},{\bm h}}({\gq})|,
\end{align*}
where the last inequality follows on noting that for any ${\bm y}\in [q]^n$ setting $\gq_i(r)=\delta(y_i,r)$ one has $\gq\in \cP([q])^n$. Since the RHS above is $o(n)$ by \eqref{eq:entry_small}, the proof of the lemma is complete.
\end{proof}
}

\medskip

\noindent
For the remaining of this section and the next, without loss of generality we will assume that diagonal elements of $A_n$ are $0$, and 
$
\max_{i,j\in [n]}|A_n(i,j)|=o(1). 
$
Next we state three lemmas which are necessary for proving Theorem \ref{thm:main}. First, for ease of writing we introduce a few  notations.
\begin{defn}\label{defn:defn}
For any ${\bm y} \in [q]^n$ define the $nq\times 1$ vector ${\bm x}:={\bm x}({\bm y})\in \cX_n$ by setting $x_{ir}:=\delta(y_i,r)$, where
\[
		\cX_n:=\left\{{\bm z} \in \{0,1\}^{nq}: \sum_{r \in [q]} z_{ir}=1 \text{ for all } i \in [n]\right\}.
		\] 
 Let $\gm:[0,1]^{nq}\mapsto [0,1]^{nq}$ by 
$$\gm_{ir}({\bm z}):= \sum_{s=1}^qJ_{rs}\sum_{j=1}^nA_n(i,j)z_{js}.$$ Note that, since diagonal entries of $A_n$ are zero, $\gm_{ir}({\bm z})$ is free of $\{z_{is},s\in [q]\}$.
Next for every $r\in [q]$, define a map $\mathsf{T}_r:(-\infty,\infty)^q\mapsto (0,1)$ by $$\mathsf{T}_r(m_1,m_2,\cdots,m_q):=\frac{e^{m_r}}{\sum_{s\in [q]}e^{m_s}}.$$  Define another $n q\times 1$ vector $\wh{\bm x}$ by
\[
\hat{ x}_{ir}:=\P_{\mu_{n}}(Y_i=r\mid Y_k=y_k,k\neq i)=\mathsf{T}_r(\gm_{i1}+h_1,\cdots,\gm_{iq}+h_q )=\frac{\exp\left(\gm_{ir}({\bm x})+h_r\right)}{\sum_{s=1}^q \exp\left(\gm_{is}({\bm x})+h_s\right)}, 
\]
and note that $\hat{\bm x} \in \hat{\cX}_n$, where 
			\[
			\hat{\cX}_n:=\left\{{\bm z} \in (0,1)^{nq}: \sum_{r \in [q]} z_{ir}=1 \text{ for all } i \in [n]\right\},
			\]
When ${\bm Y}:=(Y_i)_{i \in [n]} \sim \mu_n$, 
let ${\bm X}, \wh{\bm X}$ denote the corresponding random vectors.
Finally by a slight abuse of notation for any ${\bm z}\in [0,1]^{nq}$ let $H_n^{{\bm J},{\bm h}}({\bm z})$ stand for $F_n({\bm z})+\sum_{i\in [n],r\in [q]}h_rz_{ir}$, where $F_n:[0,1]^{nq}\mapsto \R$ is defined by
\[
F_n({\bm z}):= \f{1}{2}\sum_{r,s\in [q],i,j\in [n]}J_{rs}z_{ir}z_{js}A_n(i,j)=\f{1}{2}\sum_{i\in [n],r\in [q]}\gm_{ir}({\bm z})z_{ir}=\f{1}{2}\sum_{r,s\in [q]}J_{rs}{\bm z}_{r}'A_n{\bm z}_s,
\]
${\bm z}_r:=(z_{ir})_{1\le i\le n}\in\R^n$, {and ${\bm z}_r'$ denotes the transpose of ${\bm z}_r$}. In this notation $H_n^{{\bm J},{\bm h}}({\bm x}({\bm y}))$ is the Hamiltonian of the Potts model at ${\bm y}\in[q]^n$ in  \eqref{eq:potts_general}.

\end{defn}

\medskip

With this notation we have

\begin{lem}\label{lem:lemma1}	
If $A_n$ satisfies the mean-field assumption, then 
\[\E_{\mu_n}\left\{\left[F_n({\bm X})-F_n(\wh{{\bm X}})\right]^2\right\}=o(n^2).
\]
\end{lem}

\begin{lem}\label{lem:lemma2}
If $A_n$ satisfies the mean-field assumption, then 
\beq\label{eq:second_T1}
\E_{\mu_n}\left[\left(\sum_{i\in [n],r\in [q]} (X_{ir} - \hat{X}_{ir}) \gm_{ir}({\bm X})\right)^2\right]=o(n^2),
\eeq
{and,
\beq\label{eq:second_T2}
\E_{\mu_n}\left[\sum_{r\in [q]}\left(\sum_{i\in [n]} (X_{ir} - \hat{X}_{ir})\right)^2\right]=o(n^2).
\eeq}
\end{lem}

\medskip

Recalling the definition of net (see Definition \ref{defn:net}) we now state our next lemma. 	\begin{lem}\label{lem:mean_field}
If $A_n$ satisfies the mean-field assumption, then given any $\varepsilon>0$, there exists a $\sqrt{n}\varepsilon$-net $\mathcal{D}_n(\varepsilon)$ of the set  $\{A_n{\bm v}:{\bm v}\in [0,1]^n\}$, such that 
\begin{align}\label{eq:chatterjee}
	\lim_{n\rightarrow\infty}\frac{1}{n}\log \left|\mathcal{D}_n(\varepsilon)\right|=0.
	\end{align}

\end{lem}

\medskip

We now complete the proof of Theorem \ref{thm:main} using Lemma \ref{lem:lemma1}, Lemma \ref{lem:lemma2}, and Lemma \ref{lem:mean_field}, deferring the proof of the lemmas to Section \ref{sec:Pol}.

	\begin{proof}[Proof of Theorem \ref{thm:main}]

		For ${\bm z} \in [0,1]^{nq}$, and ${\bm w} \in (0,1)^{nq}$ define
		\[
		g_n({\bm z},{\bm w}):=\sum_{i\in[n]}\sum_{r \in [q]} z_{ir} \log w_{ir},\quad I_n({\bm z}):=g_n({\bm z},{\bm z}).
		\]
		
		Note that 
		\[
		g_n({\bm x}, \wh{\bm x})-I_n(\hat{\bm x})= \sum_{i\in [n],r \in [q]} {(x_{ir}-\hat{x}_{ir})}\log \hat{x}_{ir}=\sum_{i\in [n],r \in [q]} (x_{ir} -\hat{x}_{ir})(\gm_{ir}({\bm x})+h_r) - \sum_{i\in [n],r \in [q]} (x_{ir} -\hat{x}_{ir})\log \sigma_i,
		\]
		where
		\[
		\sigma_i:= \sum_{s\in [q]} \exp(\gm_{is}({\bm x})+h_s).
		\]
		Since for each $i \in [n]$,
		\[
		\sum_{r \in [q]} x_{ir}=\sum_{r \in [q]}\hat{x}_{ir}=1,
		\]
		we have that
		\[
		g_n({\bm x}, \wh{\bm x})-I_n(\hat{\bm x})= \sum_{i\in [n],r \in [q]} (x_{ir} -\hat{x}_{ir})(\gm_{ir}({\bm x})+h_r).
		\]
		Therefore, from Lemma \ref{lem:lemma2} we deduce that 
		\begin{align*}
		\E_{\mu_n}\left[\left(g_n({\bm X},\wh{\bm X})-I_n(\wh{\bm X})\right)^2\right]& \le 2 \E_{\mu_n}\Big[\Big(\sum_{i\in [n],r\in [q]}(X_{ir}-\hat{X}_{ir})\gm_{ir}({\bm X})\Big)^2\Big]\\
		& \qquad \qquad \qquad \qquad +2\norm{\bm h}_\infty^2q\E_{\mu_n}\sum_{r\in [q]}\Big[\Big(\sum_{i\in [n]}X_{ir}-\hat{X}_{ir}\Big)^2\Big] =o(n^2),
		\end{align*}
where we recall ${\norm{\bm h}}_\infty=\max_{r \in [q]}|h_r|$. Similarly, recalling that  ${\bm H}_n^{{\bm J},{\bm h}}({\bm z})=F_n({\bm z})+\sum_{i \in [n],r\in [q]} h_r z_{ir}$, combining Lemma \ref{lem:lemma1}, and Lemma \ref{lem:lemma2}, we get
		\begin{align*}
		\E_{\mu_n}\left[\left({\bm H}_n^{{\bm J},{\bm h}}({\bm X})- {\bm H}_n^{{\bm J},{\bm h}}(\wh{\bm X})\right)^2\right] & \le 2\E_{\mu_n}\left[\left(F_n({\bm X})-F_n(\wh{\bm X})\right)^2\right]\\
		&  \qquad \qquad \qquad +2\norm{{\bm h}}_\infty^2q\E_{\mu_n}\left[\sum_{r\in [q]}\left(\sum_{i\in [n]}X_{ir}-\wh{ X}_{ir}\right)^2\right]= o(n^2).
		\end{align*}
		Hence, applying Markov's inequality we see that $\P_{\mu_n}({\bm X}\in \cA_n) \ge 1/2$, where  
		\[
		\cA_n:= \left\{{\bm x} \in \cX_n: |{\bm H}_n({\bm x})- {\bm H}_n(\wh{\bm x})| , |g_n({\bm x}, \wh{\bm x})-I_n(\hat{\bm x})|\le \delta_n/2 \right\},
		\]
		for some $\delta_n=o(n)$, and
		
		This implies that
		\begin{align}\label{eq:phi_simplify1}
			\Phi_n({\bm J}, {\bm h}) & \le\log 2+ \log \left(\sum_{{\bm x} \in \cA_n} \exp({\bm H}_n^{{\bm J},{\bm h}}({\bm x}))\right) \notag\\
			& \le  \log 2+ \delta_n + \log \left(\sum_{{\bm x} \in \cA_n} \exp\left[{\bm H}_n^{{\bm J},{\bm h}}(\hat{\bm x})- I_n(\hat{\bm x})+ g_n({\bm x}, \hat{\bm x})\right]\right).
		\end{align}
		Since $\delta_n=o(n)$, it is enough to upper bound the rightmost term in the RHS of \eqref{eq:phi_simplify1}. This will be done by approximating the summation over $\cA_n$, by a summation over a suitable net of $\cA_n$.  
		
		To this end, using Lemma \ref{lem:mean_field} we obtain an $\sqrt{n} \vep$-net $\cD_n(\vep)$ having a sub-exponential size, of the set $\{A_n {\bm v}, {\bm v} \in [0,1]^n\}$. 
		For any $\overline{{\bm v}}:=({\bm v}_1, {\bm v}_2, \ldots, {\bm v}_q)$ such that ${\bm v}_r\in \cD_n(\vep)$ for each $r\in [q]$, choose (if exists) a $\mathsf{v}(\overline{{\bm v}}) \in \cA_n\subset \cX_n \subset \{0,1\}^{nq}$ such that  
		$\norm{A_n\mathsf{v}_r(\overline{\bm v})-\overline{\bm v}_r}_2 \le \sqrt{n} \vep$ for all $r \in [q]$. 
Here $\mathsf{v}_r(\overline{\bm v}):=(\mathsf{v}_{ir}(\overline{\bm v}))_{i \in [n]}$.  Also for any $\mathsf{v}(\overline{\bm v})$ define $$\cD(\mathsf{v}(\overline{\bm v})):= \{{\bm x} \in \cA_n: \norm{A_n{\bm x}_r- A_n\mathsf{v}_r(\overline{\bm v})}_2 \le 2 \sqrt{n} \vep, r\in [q]\}.$$ By triangle inequality it is easy to see that $$\cA_n \subset \bigcup_{{\bm v}_r\in \cD_n(\vep),r\in [q]}  \cD(\mathsf{v}(\overline{\bm v}))=\bigcup_{\overline{\bm v}\in \cD_n^q(\vep)}\cD(\mathsf{v}(\overline{\bm v})),$$ and so
		
		\beq\label{eq:phi_simplify2}
		\sum_{{\bm x} \in \cA_n} \exp\left[{\bm H}_n^{{\bm J}, {\bm h}}(\hat{\bm x})- I_n(\hat{\bm x})+ g_n({\bm x}, \hat{\bm x})\right] \le\sum_{\overline{\bm v} \in \cD_n^q(\vep)}\sum_{{\bm x} \in \cD(\mathsf{v}(\overline{\bm v}))} \exp\left[{\bm H}_n^{{\bm J}, {\bm h}}(\hat{\bm x})- I_n(\hat{\bm x})+ g_n({\bm x}, \hat{\bm x})\right].
		\eeq
		We then claim that for any ${\bm x} \in \cD(\mathsf{v}(\overline{\bm v}))$,
		{\begin{align}\label{eq:phi_simplify3}
				&|{\bm H}_n^{{\bm J}, {\bm h}}(\hat{\bm x})- {\bm H}_n^{{\bm J}, {\bm h}}(\wh{\mathsf{v}(\overline{\bm v}}))|+|g_n({\bm x}, \hat{\bm x}) - g_n({\bm x}, \wh{\mathsf{v}(\overline{\bm v}}))|+ |I_n(\hat{\bm x})- I_n(\wh{\mathsf{v}(\overline{\bm v}}))|\notag\\
				&  \qquad \qquad \qquad \qquad \qquad \qquad \le
				(q\norm{\bm h}_\infty+1)\delta_n + 2 q^2\norm{\bm J}_\infty (\norm{\bm h}_\infty+1) n  \vep+4q^3\norm{\bm J}_\infty n\vep. 
			\end{align}} 
			Since $\delta_n=o(n)$ the RHS of \eqref{eq:phi_simplify3} is bounded by $C(q)n\vep$ for some finite constant $C(q)$, for all large $n$. Thus  using \eqref{eq:phi_simplify1}-\eqref{eq:phi_simplify3} and noting the fact that
			\[
			\sum_{{\bm x} \in \cX_n} e^{g_n({\bm x}, {\bm z})}=1,
			\]
			for any ${\bm z}\in \hat{\cX}_n$,
			we deduce that
			
			\begin{align}
				\Phi_n({\bm J}, {\bm h}) &\le \log 2+ C(q)n\vep+ \log \left(\sum_{\overline{\bm v} \in \cD_n^q(\vep)}\sum_{{\bm x} \in \cD(\mathsf{v}(\overline{\bm v}))}\exp\left[{\bm H}_n^{{\bm J}, {\bm h}}(\widehat{\mathsf{v}(\overline{\bm v})})- I_n(\widehat{\mathsf{v}(\overline{\bm v})})+ g_n({\bm x}, \widehat{\mathsf{v}(\overline{\bm v})})\right]\right)\notag\\
				& \le \log 2+ C(q)n\vep +\log \left(\sum_{\mathsf{v}(\overline{\bm v}) \in \cD_n^{q}(\vep)} \exp\left[{\bm H}_n(\widehat{\mathsf{v}(\overline{\bm v})})- I_n(\widehat{\mathsf{v}(\overline{\bm v})})\right]\right)\notag\\
				& \le \log 2+ C(q)n\vep+q \log |\cD_n(\vep)| + \sup_{\gq \in \cP([q])^n} {\mathsf M}_n^{{\bm J},{\bm h}}(\gq),
			\end{align}
			where the last inequality uses the fact that for any ${\bm x}\in \cX_n$ setting $\gq_i(r)=\hat{x}_{ir}$ one has $\gq_i\in \cP([q])$, for each $i\in [n]$.
			Thus using the fact that $\log |\cD_n(\vep)|=o(n)$ we have
			\[
			\limsup_{n \ra \infty} \f{1}{n}\left[\Phi_n({\bm J}, {\bm h})- \sup_{\gq \in \cP([q])^n} {\mathsf M}_n^{{\bm J},{\bm h}}(\gq)\right] \le C(q)  \vep.
			\]
			Since $\vep>0$ is arbitrary, letting $\vep \to 0$ the proof completes. Therefore it only remains to verify \eqref{eq:phi_simplify3}. 
			
			Turning to prove \eqref{eq:phi_simplify3}, we recall that for every $\mathsf{v}(\overline{\bm v}) \in \cD_n^{(q)}(\vep)$, and any ${\bm x} \in \cD(\mathsf{v}(\overline{\bm v}))$, both ${\bm x}$ and $\mathsf{v}(\overline{\bm v})$ are in the set $\cA_n$. Therefore
			\[
			|{\bm H}_n^{{\bm J}, {\bm h}}({\bm x})- {\bm H}_n^{{\bm J}, {\bm h}}(\hat{\bm x})| \le \delta_n/2, \qquad \text{ and } \qquad |{\bm H}_n^{{\bm J}, {\bm h}}({\mathsf{v}(\overline{\bm v})})- {\bm H}_n^{{\bm J}, {\bm h}}(\widehat{\mathsf{v}(\overline{\bm v})})| \le \delta_n/2.
			\]
			Thus, in order to bound $|{\bm H}_n^{{\bm J}, {\bm h}}(\hat{\bm x})- {\bm H}_n^{{\bm J}, {\bm h}}(\widehat{\mathsf{v}(\overline{\bm v})})|$, we only need to consider $|{\bm H}_n^{{\bm J}, {\bm h}}({\bm x})- {\bm H}_n^{{\bm J}, {\bm h}}(\mathsf{v}(\overline{\bm v}))|$. Now recall that 
			\[{\bm H}_n^{{\bm J}, {\bm h}} ({\bm x}) = \f{1}{2}\sum_{r,s \in [q]}J_{rs}{\bm x}_r'A_n {\bm x}_s + \sum_{r \in [q]} h_r {\bm 1}' {\bm x}_r.\] 
Note that ${\bm x} \in \cD(\mathsf{v}(\overline{\bm v}))$  we have
			\begin{align}
				\left|{\bm x}_r'A_n {\bm x}_s- \mathsf{v}_r(\overline{\bm v})'A_n \mathsf{v}_s(\overline{\bm v})\right|& \le \left|{\bm x}_r'A_n {\bm x}_s- {\bm x}_r'A_n \mathsf{v}_s(\overline{\bm v})\right| + \left|{\bm x}_r'A_n \mathsf{v}_s(\overline{\bm v}) - \mathsf{v}_r(\overline{\bm v})'A_n \mathsf{v}_s(\overline{\bm v})\right| \notag\\
				& \le \sqrt{n} \norm{A_n{\bm x}_r - A_n\mathsf{v}_r(\overline{\bm v})}_2+ \sqrt{n} \norm{A_n{\bm x}_s - A_n\mathsf{v}_s(\overline{\bm v})}_2 \le 4 n \vep.
			\end{align}
			Next we proceed to bound $|{\bm 1}' {\bm x}_r - {\bm 1}'\mathsf{v}_r(\overline{\bm v})|$. From Lemma \ref{lem:lemma2}, applying Markov's inequality it follows that $|{\bm 1}'{\bm x}_r- {\bm 1}'\hat{\bm x}_r| \le \delta_n/2$, for every ${\bm x} \in \cA_n$, and $r \in [q]$. Hence it remains to find an upper bound on $\norm{\hat{\bm x}_r- \widehat{\mathsf{v}_r(\overline{\bm v})}}_2$. To this end, recalling that $\hat{x}_{ir}=\mathsf{T}_r(\gm_{i1}({\bm x})+h_1,\cdots,\gm_{iq}({\bm x})+h_q)$ and noting that
			\beq\label{eq:derivative}
			\norm{\frac{\partial {\sf T}_r(m_1,\cdots,m_q)}{\partial m_s}}_\infty=\norm{{\sf T}_r(m_1,\cdots,m_q)\{\delta(r,s)- {\sf T}_s(m_1,\cdots,m_q)\}}_\infty\le 1,
			\eeq			
applying a multivariate version of the mean-value theorem we obtain that
			\begin{align*}
				|\hat{x}_{ir}-\wh{\mathsf{v}(\overline{\bm v})}_{ir}| \le \sum_{s \in [q]}|\gm_{is}({\bm x}) - \gm_{is}({\mathsf{v}(\overline{\bm v})})| & = \sum_{s \in [q]}\left| \sum_{s' \in [q]} J_{ss'}\left\{ (A_n {\bm x}_{s'})_i - (A_n \mathsf{v}_{s'}(\overline{\bm v}))_i\right\}\right|\\
				& \le q \norm{\bm J}_\infty \sum_{s' \in [q]}\left| (A_n {\bm x}_{s'})_i - (A_n \mathsf{v}_{s'}(\overline{\bm v}))_i\right|.
			\end{align*}This further implies that
			\[
			\norm{\hat{\bm x}_r- \widehat{\mathsf{v}(\overline{\bm v})}}_2^2 \le q^3 \norm{{\bm J}}_\infty^2 \sum_{i \in[n], s'\in[q]} \left| (A_n {\bm x}_{s'})_i - (A_n \mathsf{v}_{s'}(\overline{\bm v}))_i\right|^2\le 4 q^4 \norm{{\bm J}}_\infty^2 n\vep^2.
			\]
			Therefore,
			\begin{align}\label{eq:h_n_bound}
				|{\bm H}_n^{{\bm J}, {\bm h}}(\hat{\bm x})- {\bm H}_n^{{\bm J}, {\bm h}}(\widehat{\mathsf{v}(\overline{\bm v})})| & \le |{\bm H}_n^{{\bm J}, {\bm h}}({\bm x})- {\bm H}_n^{{\bm J}, {\bm h}}(\mathsf{v}(\overline{\bm v}))|+ \delta_n\notag\\
				& \le \delta_n+ \f{\norm{\bm J}_\infty}{2}\sum_{r,s \in [q]}\left|{\bm x}_r'A_n {\bm x}_s- \mathsf{v}_r(\overline{\bm v})A_n \mathsf{v}_s(\overline{\bm v})\right| +  \norm{\bm h}_\infty\sum_{r \in [q]}\left|{\bm 1}'{\bm x}_r- {\bm 1}'\mathsf{v}_r(\overline{\bm v})\right|\notag\\
				& \le (q\norm{\bm h}_\infty+1)\delta_n + 2 q^2\norm{\bm J}_\infty n \vep+ \sqrt{n} \norm{\bm h}_\infty\sum_{r \in [q]}\norm{\hat{\bm x}_r- \wh{\mathsf{v}_r(\overline{\bm v})}}_2 \notag\\
				& \le (q\norm{\bm h}_\infty+1)\delta_n + 2 q^2\norm{\bm J}_\infty (\norm{\bm h}_\infty+1) n  \vep.
			\end{align} 
			Next we proceed to bound $|g_n({\bm x}, \hat{\bm x}) - g_n({\bm x}, \widehat{\mathsf{v}(\overline{\bm v})})|$. To this end, we have
			\begin{align*}
			&|g_n({\bm x},\hat{{\bm x}})-g_n({\bm x},\wh{\mathsf{v}(\overline{\bm v})})|\\
			\le &\sum_{i\in [n],r\in [q]}\Big|\log {\sf T}_r(\gm_{i1}({\bm x})+h_1,\cdots,\gm_{iq}({\bm x})+h_q)-\log {\sf T}_r(\gm_{i1}(\wh{\mathsf{v}(\overline{\bm v})})+h_1,\cdots,\gm_{iq}(\wh{\mathsf{v}(\overline{\bm v})})+h_q)\Big|\\
			\le &\sum_{i\in [n],r,s\in [q]}|\gm_{is}({\bm x})-\gm_{is}({\mathsf{v}(\overline{\bm v})})|\norm{\frac{\partial \log {\sf T}_r}{\partial m_s}}_\infty\\
			\le&q\sum_{i\in [n],s\in [q]}|\gm_{is}({\bm x})-\gm_{is}({\mathsf{v}(\overline{\bm v})})|,
			\end{align*}
			where the last inequality uses \eqref{eq:derivative} to conclude that $\norm{\frac{\partial \log {\sf T}_r}{\partial m_s}}_\infty\le 1$. 

\noindent
This gives
			\begin{align}\label{eq:g_n_bound}
				|g_n({\bm x}, \hat{\bm x}) - g_n({\bm x}, \widehat{\mathsf{v}(\overline{\bm v})})| &\le q\sum_{i\in [n]s\in [q]}|\gm_{is}({\bm x}) - \gm_{is}({\mathsf{v}(\overline{\bm v})})|  \notag\\
&= q \sum_{i\in [n], s \in [q]}\left| \sum_{s' \in [q]} J_{ss'}\left\{ (A_n {\bm x}_{s'})_i - (A_n \mathsf{v}_{s'}(\overline{\bm v}))_i\right\}\right|\notag\\
				& \le  q^2 \norm{\bm J}_\infty \sum_{i \in [n], s' \in [q]}\left| (A_n {\bm x}_{s'})_i - (A_n \mathsf{v}_{s'}(\overline{\bm v}))_i\right|\notag\\
				& \le  q^2\sqrt{n} \norm{\bm J}_\infty\sum_{s' \in [q]}\norm{A_n{\bm x}_{s'}-A_n\mathsf{v}_{s'}(\overline{{\bm v}})}_2 \le 2q^3 \norm{\bm J}_\infty n \vep,
			\end{align}
			where the penultimate step uses Cauchy-Schwarz inequality, and the last step uses the fact that ${\bm x}\in \cD({\mathsf{v}}(\ol{\bm v}))$.

			Now it remains to bound $ |I_n(\hat{\bm x})- I_n(\widehat{\mathsf{v}(\overline{\bm v})})|$, for which we follow a similar program. Setting $\gamma(t):=t\log t$ for $t\ge 0$, we have
\begin{align*}
&|I_n(\hat{\bm x})- I_n(\widehat{\mathsf{v}(\overline{\bm v})})|\\
\le &\sum_{i\in [n],r\in [q]}\Big|\gamma\Big(\mathsf{T}_r(\gm_{i1}({\bm x})+h_1,\cdots,\gm_{iq}({\bm x})+h_q)\Big)-\gamma\Big(\mathsf{T}_r(\gm_{i1}(\mathsf{v}(\overline{\bm v}))+h_1,\cdots,\gm_{iq}(\mathsf{v}(\overline{\bm v}))+h_q)\Big)\Big|\\
\le &\sum_{i\in [n],r,s\in [q]}|\gm_{is}({\bm x})-\gm_{is}(\mathsf{v}(\overline{\bm v}))|\norm{\frac{\partial (\gamma\circ {\sf T}_r)}{\partial m_s}}_\infty
\end{align*}			
 Using \eqref{eq:derivative} gives
\begin{align*}
\norm{\frac{\partial (\gamma\circ \mathsf{T}_r)}{\partial m_s}}_\infty=\norm{\mathsf{T}_r(1+\log \mathsf{T}_r)\{\delta(r,s)-\mathsf{T}_s\}}_\infty\le \sup_{t\in [0,1]}t|1+\log t|\le 1,
\end{align*}			  
			 and therefore 	  

			\beq\label{eq:i_n_bound}
			|I_n(\hat{\bm x})- I_n(\widehat{\mathsf{v}(\overline{\bm v})})| \le q\sum_{i\in [n]s\in [q]}|\gm_{is}({\bm x}) - \gm_{is}(\wh{\mathsf{v}(\overline{\bm v})})| \le 2q^3 \norm{\bm J}_\infty n \vep,
			\eeq
			where the last bound follows by arguments similar to \eqref{eq:g_n_bound}.
			Finally combining \eqref{eq:h_n_bound}-\eqref{eq:i_n_bound} we arrive at \eqref{eq:phi_simplify3}, and this completes the proof.
		\end{proof}

\section{Proof of auxiliary Lemmas}\label{sec:Pol}
{

\noindent
In this section we prove Lemma \ref{lem:lemma1}, Lemma \ref{lem:lemma2}, and Lemma \ref{lem:mean_field}. We start with the proof of Lemma \ref{lem:lemma1}.

\begin{proof}[Proof of {Lemma} \ref{lem:lemma1}]
To lighten the notation, we drop the subscript $n$ in $F_n$, and write $F$ through out the proof. Before we begin the proof let us introduce some notation:
\[
F_{ir}({\bm x}):= \f{\partial}{\partial x_{ir}}F({\bm x}), \quad \text{ and }  \quad F_{ir, js}({\bm x}):= \f{\partial^2}{\partial x_{ir} \partial x_{js}}F({\bm x}).
\]
 Equipped with these notation by mean-value theorem we have
 \[
 F({\bm x})- F(\wh{{\bm x}})= \int_{0}^1\sum_{i\in [n],r\in [q]} (x_{ir}-\hat{x}_{ir}) F_{ir}(t{\bm x}+ (1-t)\wh{{\bm x}}) dt.
 \]
 Thus denoting $\Delta({\bm x}):= F({\bm x})-F(\wh{{\bm x}})$, and $u_{ir}(t,{\bm x}):= F_{ir}(t{\bm x}+ (1-t)\wh{{\bm x}})$, we have
 \beq\label{eq:difference_simplify}
 \E_{\mu_n}\left\{\left[F({\bm X})-F(\wh{{\bm X}})\right]^2\right\} = \int_0^1 \sum_{i\in [n],r\in [q]} \E_{\mu_n}\left((X_{ir}-\hat{X}_{ir}) u_{ir}(t,{\bm X})\Delta({\bm X})\right)dt.
\eeq
Hence to complete the proof it is enough to find upper bound on  the {RHS} of \eqref{eq:difference_simplify} for each value of $t \in [0,1]$. To this end observe that for any $i \in [n],r\in [q]$  we have
\[
\E_{\mu_n}\left((X_{ir} - \wh{X}_{ir}) u_{ir}(t,{\bm X}^{(ir)})\Delta({\bm X}^{(ir)})\right)=0,
\]
where ${\bm X}^{(ir)}$ is obtained by setting $X_{ir}=0$ in the random vector ${\bm X}$.
Therefore for each $i \in [n],r\in [q]$ it suffices to consider the difference
\[
\E_{\mu_n}\left((X_{ir}-\wh{X}_{ir}) u_{ir}(t,{\bm X})\Delta({\bm X})-(X_{ir} - \hat{X}_{ir}) u_{ir}(t,{\bm X}^{(ir)})\Delta({\bm X}^{(ir)}) \right),
\]
and show that
\beq\label{eq:bound4.1}
\sup_{t\in [0,1]}\left|\sum_{i\in [n],r\in [q]}\E_{\mu_n}\left((X_{ir}-\wh{X}_{ir}) (u_{ir}(t,{\bm X})- u_{ir}(t,{\bm X}^{(ir)}))\Delta({\bm X}^{(ir)}) \right)\right|=o(n^2),
\eeq
 and
 \beq\label{eq:bound4.2}
\sup_{t\in [0,1]}\left|\sum_{i\in [n],r\in [q]} \E_{\mu_n}\left((X_{ir} - \hat{X}_{ir}) u_{ir}(t,{\bm X})(\Delta({\bm X}) - \Delta({\bm X}^{(ir)})) \right)\right|=o(n^2).
\eeq
To establish \eqref{eq:bound4.1}, we first note that using \eqref{eq:model_assumption2} there exists $C_1<\infty$ such that
\beq\label{eq:intermediate}
\left|\E_{\mu_n}\left((X_{ir}-\hat{X}_{ir}) (u_{ir}(t,{\bm X})- u_{ir}(t,{\bm X}^{(ir)}))\Delta({\bm X}^{(ir)}) \right)\right| \le  C_1n\E_{\mu_n}\left|u_{ir}(t,{\bm X})- u_{ir}(t,{\bm X}^{(ir)})\right|.
\eeq
Since 
\begin{align*}
u_{ir}(t,X)&={t}\gm_{ir}({\bm x})+({1-t}){\gm}_{ir}(\hat{\bm x})\\
&={t}\gm_{ir}({\bm x})+({1-t})\sum_{j\in [n],s\in [q]}J_{rs}A_n(i,j)\mathsf{T}_s(\gm_{j1}({\bm x})+h_1,\cdots,\gm_{jq}({\bm x})+h_q),
\end{align*} 
and $\gm_{ir}({\bm x})$ is free of $\{x_{is}\}_{s\in [q]}$, by chain rule the RHS of \eqref{eq:intermediate} can be bounded by
\begin{align*}
&C_1n\Big|\sum_{j\in [n],s\in [q]} J_{rs}A_n(i,j)\sum_{s'\in [q]}\Big(\frac{\partial {\sf T}_s}{\partial m_{s'}}(m_1,\cdots,m_q)\Big|_{\{m_{r'}=\gm_{jr'}({\bm x})+h_{r'},r'\in [q]\}}\Big) A_n(i,j)J_{rr'}\Big|\\
 \le & C_1q^2n\|{\bm J}\|_\infty^2 \sum_{j\in [n]} |A_n(i,j)|^2, 
\end{align*}
where the last step uses \eqref{eq:derivative}. This, on
 summing over $i\in [n]$, and $r\in [q]$ gives \eqref{eq:bound4.1} by the mean-field assumption.
 Next turning to bound \eqref{eq:bound4.2},  
we first write
\begin{align}
\notag2(\Delta({\bm X})-\Delta({\bm X}^{(ir)}))=&\sum_{a,b\in [q]}J_{ab}\left[ {\bm X}_a'A_n {\bm X}_b - {\bm X}_a^{(ir)'}A_n {\bm X}_b^{(ir)}\right]-
 \sum_{a,b\in [q]}J_{ab}\left[\wh{\bm X}_a' A_n \wh{\bm X}_b- \wh{\bm X}_a' A_n \Big(\wh{{\bm X}^{(ir)}}\Big)_b\right]\notag\\
& -\sum_{a,b\in [q]}J_{ab}\left[\wh{\bm X}_a' A_n \wh{{\bm X}^{(ir)}}_b- \wh{{\bm X}^{(ir)}}'_a A_n \Big(\wh{{\bm X}^{(ir)}}\Big)_b\right].
 \label{eq:split_delta}
\end{align}
Here the notation $\Big(\wh{{\bm X}^{(ir)}}\Big)_b$ means the $b^{th}$ column of the matrix $\wh{{\bm X}^{(ir)}}$. Now note for any $a \in [q]\setminus\{r\}$, ${\bm x}_a={\bm x}_a^{(ir)}$, 
\[
{\bm x}_a'A_n {\bm x}_r - {\bm x}_a^{(ir)'}A_n {\bm x}_r^{(ir)} = {\bm x}_a'A_n {\bm x}_r - {\bm x}_a A_n {\bm x}_r^{(ir)} = x_{ir} \left(A_n {\bm x}_a\right)_i,
\]
and
\[
{\bm x}_r'A_n {\bm x}_r - {\bm x}_r^{(ir)'}A_n {\bm x}_r^{(ir)}= x_{ir} \left(A_n {\bm x}_r\right)_i + x_{ir} \left(A_n {\bm x}_r^{(ir)}\right)_i= 2 x_{ir} \left(A_n {\bm x}_r\right)_i,
\]
where the last equality follows from the fact that $A_n(i,i)=0$. Thus recalling the definition of $\gm_{ir}({\bm x})$, we have 

\begin{eqnarray}\label{eq:term1_delta_split}
&& \left|\sum_{i\in [n],a,b,r\in [q]}{J_{ab}}\left((X_{ir} - \hat{X}_{ir}) u_{ir}(t,{\bm x})\{{\bm x}_a'A_n {\bm x}_b - {\bm x}_a^{(ir)'}A_n {\bm x}_b^{(ir)}\} \right)\right|\notag\\
 &\le &2\sum_{i\in [n],r\in [q]}  |\gm_{ir}({\bm x})|(|\gm_{ir}({\bm x})|+|{\gm}_{ir}(\hat{\bm x})|)\notag\\
 &\le &2\left\{\sum_{r\in [q],i\in [n]}\gm_{ir}({\bm x})^2+\sum_{r\in [q]}\sqrt{\sum_{i=1}^n\gm_{ir}({\bm x})^2}\sqrt{\sum_{i=1}^n{\gm_{ir}(\hat{\bm x})}^2}\right\},
\end{eqnarray}
where the first step uses the fact that $|u_{ir}(t,{\bm x})| \le t |\gm_{ir}({\bm x})+ (1-t)|\gm_{ir}(\hat{\bm x})$, and last step follows by an application of Cauchy-Schwarz inequality.

\noindent
Also the mean-field assumption implies that $\lambda_{\max}(A_n)= o(\sqrt{n})$, and therefore we have 
$$\sum_{i\in [n]}\gm_{ir}({\bm x})^2\le q\norm{{\bm J}}_\infty^2\sum_{s\in [q]}\norm{A_n{\bm x}_s}_2^2\le q\norm{{\bm J}}_\infty^2\lambda_{\max}^2(A_n)\sum_{s \in [q]}\norm{{\bm x}_s}_2^2=o(n^2).$$
By similar arguments, from \eqref{eq:term1_delta_split} we deduce that 
\beq\label{eq:combine_1}
\sup_{t\in [0,1]}\left|\sum_{i\in [n],r\in [q]}\E_{\mu_n}\left((X_{ir} - \wh{X}_{ir}) u_{ir}(t,{\bm X})\sum_{a,b\in [q]}J_{ab}\left[ {\bm X}_a'A_n {\bm X}_b - {\bm X}_a^{(ir)'}A_n {\bm X}_b^{(ir)}\right] \right)\right| = o(n^2).
\eeq
 Next we consider the second term in the RHS of \eqref{eq:split_delta}, where using  first order Taylor's theorem, upon application of chain rule, followed by \eqref{eq:derivative}, gives
\begin{align*}
\wh{\bm x}_a' A_n \wh{\bm x}_r- \wh{\bm x}_a' A_n \Big(\wh{\bm x^{(ir)}}\Big)_b= &\sum_{j,k\in [n]} \hat{x}_{ja}A_n(j,k)\Big\{\hat{x}_{kr} -\Big(\widehat{x^{(ir)}}\Big)_{kb}\Big\}\\
= & {\sum_{j,k\in [n]}} \hat{x}_{ja}A_n(j,k)A_n(i,k)\xi_{i,k,b,r},
\end{align*}
 for some $\xi_{i,k,b,r}$, such that $|\xi_{i,k,b,r}|\le q \norm{{\bm J}}_\infty$. 
Denoting $\norm{A_n}_\infty:=\sup_{i,j}|A_n(i,j)|$, and summing over $i\in [n],a,b,r\in [q]$ this gives

\begin{eqnarray*}
&& \left|\sum_{a,b,r\in [q],i\in [n]}J_{ab} \left((x_{ir} - \hat{x}_{ir}) u_{ir}(t,{\bm x})\Big\{\wh{\bm x}'_aA_n \hat{\bm x}_b - \hat{\bm x}'_aA_n \Big(\wh{{\bm x}^{(ir)}}\Big)_b\Big\} \right)\right|\notag\\
 &=&\left|\sum_{i,k\in [n],b,r\in [q]} \left((x_{ir} - \hat{x}_{ir})x_{ir} {\gm}_{kb}(\hat{\bm x}) A_n(i,k)\xi_{i,k,b,r} u_{ir}(t,{\bm x}) \right)\right|\notag\\
 & \le&q\norm{A_n}_\infty \norm{J}_\infty\sum_{i,k\in [n],b,r\in [q]}(|\gm_{ir}({\bm x})|+|{\gm}_{ir}(\hat{\bm x})|)|{\gm}_{kb}(\hat{\bm x})|\\
 &=&q\norm{A_n}_\infty \norm{J}_\infty\bigg[\left(\sum_{i\in [n],r\in [q]}|\gm_{ir}({\bm x})|\right)\left(\sum_{k\in [n],b\in [q]}|{\gm}_{kb}(\hat{\bm x})|\right)\\
 & & \qquad \qquad \qquad \qquad \qquad +\left(\sum_{i\in [n],r\in [q]}|{\gm}_{ir}(\hat{\bm x})|\right)\left(\sum_{k\in [n],b\in [q]}|{\gm}_{kb}(\hat{\bm x})|\right)\bigg].
\end{eqnarray*}
Now using \eqref{eq:model_assumption2} we obtain
\[
\sum_{i\in [n],r\in [q]}|\gm_{ir}({\bm x})|,\sum_{i\in [n],r\in [q]}|{\gm}_{ir}(\hat{\bm x})|= O(n).
\]
This together with the fact that $\norm{A_n}_\infty=o(1)$, implies that the RHS above is $o(n^2)$, thus giving
\beq\label{eq:combine_2}
\sup_{t \in [0,1]}\left|\sum_{i\in [n],r\in [q]} \E_{\mu_n}\left((X_{ir} - \hat{X}_{ir}) u_{ir}(t,{\bm X}) \sum_{a,b\in [q]}J_{ab}\left[\wh{\bm X}_a' A_n \wh{\bm X}_b- \wh{\bm X}_a' A_n \Big(\wh{{\bm X}^{(ir)}}\Big)_b\right]\right)\right|= o(n^2).
\eeq
Finally, considering the third term in the RHS of \eqref{eq:split_delta} and using  first order Taylor's theorem again, we also note that
\begin{align*}
\wh{\bm x}_b' A_n \Big(\wh{{\bm x}^{(ir)}}\Big)_a- \Big(\wh{{\bm x}^{(ir)}}\Big)_b' A_n \Big(\wh{{\bm x}^{(ir)}}\Big)_a= &\sum_{j,k\in [n]}\Big(\widehat{x^{(ir)}}\Big)_{kb} (\hat{x}_{ja} -\Big(\widehat{x^{(ir)}})\Big)_{ja}A_n(j,k)\\
= & \sum_{j,k\in [n]}\Big(\widehat{x^{(ir)}}\Big)_{kb} \xi_{i,j,a,r}A_n(j,k)A_n(i,k).
\end{align*}
From this, proceeding similarly as in the proof of\eqref{eq:combine_2} we have
\begin{eqnarray}\label{eq:term3_delta_split}
&&\left| \sum_{i,j,k\in [n],a,b,r\in [q]}J_{ab} \left((x_{ir} - \hat{x}_{ir}) u_i(t,{\bm x})\Big\{\wh{\bm x}'_bA_n \wh{{\bm x}^{(ir)}}_a - \wh{{\bm x}^{(ir)}}'_bA_n \Big(\wh{{\bm x}^{(ir)}}\Big)_a\Big\} \right)\right|\notag\\
&=&\left|\sum_{i,j\in [n],a,r\in [q]} \left((x_{ir} - \wh{x}_{ir})x_{ir} \gm_{ja}(\wh{{\bm x}^{(ir)}}) A_n(i,k)\xi_{i,j,a,r} u_{ir}(t,{\bm x}) \right)\right|\notag\\
 & \le&q\norm{A_n}_\infty \norm{J}_\infty\sum_{i,j\in [n],a,r\in [q]}(|\gm_{ir}({\bm x})|+|{\gm}_{ir}(\hat{\bm x})|)(|\wh{\gm}_{ja}(\wh{{\bm x}^{(ir)}})|)=o(n^2)
\end{eqnarray}
as before, and so
\beq\label{eq:combine_3}
 \left|\sum_{i\in [n],r\in [q]} \E_{\mu_n}\left((X_{ir} - \hat{X}^{(ir)}) u_{ir}(t,{\bm X})\sum_{a,b\in [q]}J_{ab}\left[\wh{\bm X}_a' A_n \Big(\wh{{\bm X}^{(ir)}}\Big)_b- \Big(\wh{{\bm X}^{(ir)}}\Big)'_a A_n\Big( \wh{{\bm X}^{(ir)}}\Big)_b\right]\right)\right|= o(n^2).
\eeq
Finally combining \eqref{eq:combine_1}, \eqref{eq:combine_2}, and \eqref{eq:combine_3}, the proof is complete.
\end{proof}

\medskip

\noindent
Now we prove Lemma \ref{lem:lemma2}.

\begin{proof}[Proof of {Lemma} \ref{lem:lemma2}]
First we prove \eqref{eq:second_T1}. To this end, for any ${\bm x}\in \cX_n$ define
\[
G({\bm x}):= \sum_{i\in [n],r\in [q]} (x_{ir} - \hat{x}_{ir}) \gm_{ir}({\bm x})
\]
and note that 
\[
\E_{\mu_n}\left ((X_{ir} -\hat{X}_{ir}) \gm_{ir}({\bm X}) G({\bm X}^{(ir)})\right)=0.
\]
Thus we need to show that
\[
\sum_{i\in [n],r\in [q]} \E_{\mu_n}\left[ (X_{ir}- \hat{X}_{ir})\gm_{i,r}({\bm X})(G({\bm X})- G({\bm X}^{(ir)}))\right] =o(n^2).
\]
To this end, we first observe that
\beq\label{eq:G_diff}
G({\bm x})- G({\bm x}^{(ir)}) = 2x_{ir} \gm_{ir}({\bm x})+\sum_{j\in [n],s\in [q]} \hat{x}_{js}\left(\gm_{js}({\bm x}^{(ir)})- \gm_{js}({\bm x})\right)+ \sum_{j \in [n],s\in [q]}(\widehat{x^{(ir)}}_{js}-\hat{x}_{js}) \gm_{js}({\bm x}^{(ir)}) .
\eeq
For the first term in the RHS of \eqref{eq:G_diff}, proceeding as in \eqref{eq:combine_1}, by a Cauchy Schwarz argument we have
$$\left|\sum_{i\in [n],r\in [q]} (x_{ir} -\hat{x}_{ir})x_{ir}\gm_{ir}^2\right|\le \sum_{i\in [n],r\in [q]}\gm_{ir}({\bm x})^2=o(n^2),$$
giving
\beq\label{eq:secondT_1a}
\left|\E_{\mu_n}\sum_{i\in [n],r\in [q]} \left[(X_{ir} -\hat{X}_{ir})X_{ir}\gm_{ir}({\bm X})^2\right]\right|=o(n^2).
\eeq
For controlling the second term in the RHS of \eqref{eq:G_diff} first note that $\gm_{js}({\bm x}) - \gm_{js}({\bm x}^{(ir)}) = A_n(i,j)J_{rs}x_{ir}$. Thus proceeding as in \eqref{eq:term1_delta_split} again we further have
\begin{align*}
&\left|\sum_{i,j\in [n],r,s\in [q]} (x_{ir} -\hat{x}_{ir})\gm_{ir}({\bm x})\hat{x}_{js}(\gm_{js}({\bm x}) - \gm_{js}({\bm x}^{(ir)}))\right|\notag\\
&  =  \left| \sum_{i\in [n],r\in [q]} (x_{ir} - \hat{x}_{ir})x_{ir}\gm_{ir}({\bm x}){\gm}_{ir}(\hat{\bm x})\right| \le  \sum_{i\in [n],r\in [q]} |\gm_{ir}({\bm x})| |{\gm}_{ir}(\hat{\bm x})|
\end{align*}
which is $o(n^2)$ by a Cauchy Schwarz argument as in the proof of \eqref{eq:combine_1}. Thus we have
\begin{equation}\label{eq:secondT_1b}
\left|\E_{\mu_n}\sum_{i,j\in [n],r,s\in [q]} (X_{ir} -\wh{X}_{ir})\gm_{i,r}({\bm X})\wh{X}_{js}(\gm_{js}({\bm X}) - \gm_{js}({\bm X}^{(ir)}))\right|=o(n^2).
\end{equation}
Finally for controlling the third term in the RHS of \eqref{eq:G_diff}, applying first order Taylor's theorem yields that $\hat{x}_{js} - \Big(\widehat{x^{(ir)}}\Big)_{js} = A_n(i,j)\xi_{i,j,r,s}$ with $|\xi_{i,j,r,s}|\le q\norm{{\bm J}}_\infty$. Thus we have 
\begin{align*}
&\left|\sum_{i,j\in [n],r,s\in [q]} (x_{ir} -\hat{x}_{ir})\gm_{ir}({\bm x}) \gm_{js}({\bm x}^{(ir)})\Big\{\hat{x}_{js} -\Big(\widehat{x^{(ir)}}\Big)_{js}\Big\}\right|\notag\\
& \le q\norm{A_n}_\infty\norm{{\bm J}}_\infty \left(\sum_{i\in [n],r\in [q]} |\gm_{ir}({\bm x})|\right)^2+nq^2\norm{A_n}_\infty^2\norm{{\bm J}}_\infty^2\sum_{i\in [n],r\in [q]}|\gm_{ir}({\bm x})|,
\end{align*}
which is $o(n^2)$ by arguments similar to the proof of \eqref{eq:combine_2}. This gives
\begin{align}\label{eq:secondT_1c}
&\left|\E_{\mu_n}\sum_{i,j\in [n],r,s\in [q]} (X_{ir} -\wh{X}_{ir})\gm_{i,r}({\bm X}) \gm_{js}({\bm X}^{(ir)})\Big\{\hat{X}_{js} -\Big(\widehat{X^{(ir)}}\Big)_{js}\Big\}\right|=o(n^2),
\end{align}
which on combining  with\eqref{eq:secondT_1a} and \eqref{eq:secondT_1b} completes the  proof of \eqref{eq:second_T1}.

\medskip

\noindent
Next to {prove \eqref{eq:second_T2}}, we define
\[
\wt{G}_r({\bm x})= \sum_{i\in [n]} (x_{ir} - \hat{x}_{ir}),
\]
and therefore
\[
\wt{G}_r({\bm x})-\wt{G}_r({\bm x}^{(ir)}) = x_{ir} - \sum_{j\in[n]}\Big\{\hat{x}_{jr} -\Big(\widehat{x^{(ir)}}\Big)_{jr}\Big\}. 
\]
Thus observing that 
\[
\E_{\mu_n}\left[(X_{ir} -\hat{X}_{ir}) \wt{G}_r({\bm X}^{(ir)})\right]=0,
\]
for any $i \in [n],r\in [q]$, we only need to show that
\[
\sum_{i\in [n]}\E_{\mu_n}\left[(X_{ir} -\hat{X}_{ir}) (\wt{G}_r({\bm X})-\wt{G}_r({\bm X}^{(ir)}))\right]=o(n^2).
\]
This can be done proceeding similarly as above. We omit the details. 
\end{proof}
\medskip

\noindent
Now we prove Lemma \ref{lem:mean_field} . Before going to the proof let us introduce the following notation:

For $r\in\N$ and $R>0$ let $\cB_r(R)$ denote the Euclidean ball of radius $R$ in dimension $r$, i.e.
		$$\cB_r(R):=\{{\bm v}\in\R^r:\norm{{\bm v}}_2\le R\}.$$
		
\noindent
The proof of Lemma \ref{lem:mean_field} also requires the following standard estimate on an $\eta$-net $\cB_r(R)$. Its proof is based on simple volumetric argument. We refer the reader to \cite[Lemma 2.6]{MS} for its proof. 
	\begin{lem}\label{lem:net}
		For any $R,\eta \in \R$, and  $r\in \N$, there exists an $\eta$-net of $\cB_r(R)$ of size at most $\max\big\{1,(3R/\eta)^r\big\}$.
		\end{lem}

\medskip

\noindent
		\begin{proof}[Proof of Lemma \ref{lem:mean_field}]
		
		Let $\{\lambda_1(A_n),\cdots,\lambda_n(A_n)\}$ denote the eigenvalues of $A_n$. Fixing $\varepsilon\in (0,1)$, let $N_n$ denote the number of eigenvalues of $A_n$ which are greater than $\varepsilon/2$ in absolute value. Since $A_n$ satisfies the mean-field assumption, by Chebyshev's inequality we have
		that \begin{align}\label{eq:mean_field_used_1}
		0 \le \lim_{n\rightarrow\infty}	\frac{N_n}{n}\le \lim_{n\rightarrow\infty}\frac{4}{n\vep^2}\sum_{i\in [n]}\lambda_i(A_n)^2=0.
			\end{align}
				
			Set $\ell=\ell_n:=\lceil {\log_2 \sqrt{n}}\rceil$, and for $1\le k\le \ell$ let $I_k:=\{1\le i\le n:2^{k-1}<|\lambda_i(A_n)|\le 2^k\}$. Thus with $I_0:=\{1\le i\le n:\varepsilon/2<|\lambda_i(A_n)|\le 1\}$ and $I:=\cup_{k=0}^\ell I_k$, and using the fact that $\tr(A_n^2)=O(n)$, we have 
		
		$$\sum_{k=0}^{\ell}|I_k|=|I|=N_n.$$
		 For $0\le k,j\le \ell$, if $I_k\neq \phi$ we let $\mathcal{C}_{k}(j)$ denote an $\varepsilon 2^{-{(k+1)}}\sqrt{|I_k|}$-net of the set $\cB_{|I_k|}(2^{j})$. By Lemma \ref{lem:net} we may and will assume that 
		 \begin{align}\label{eq:bound_1}
		 	|\cC_{k}(j)|\le \max\bigg\{1,\Big(\frac{6}{\varepsilon}\Big)^{|I_k|} \Big(\frac{2^{k+j}}{\sqrt{|I_{k}|}}\Big)^{|I_{k|}}\bigg\}.
		 	\end{align}
		 Setting 
		 $$\cS_{n}(\varepsilon):= \bigcup\limits_{0\le j_0,j_1,\cdots,j_\ell\le \ell:\sum_{k=0}^\ell 2^{2j_k}\le 5n} \{\cC_0(j_0)\times \cC_1(j_1)\times \cdots \times \cC_\ell(j_\ell)\}$$
		 we first claim that 
		\begin{align}\label{eq:sub_expo}
			\lim_{n\rightarrow\infty}\frac{1}{n}\log |\cS_{n}(\varepsilon)|=0.
			\end{align}
 Deferring the proof of \eqref{eq:sub_expo} and setting
	$$\cD_n(\varepsilon):=\Big\{\sum_{i\in I}\lambda_i(A_n)c_i{\bm p}_i:{\bm c}:=(c_i)_{i\in I}\in \cS_{n}(\varepsilon)\Big\},$$
where ${\bm p}_1, {\bm p}_2, \ldots, {\bm p}_n$ are the eigenvectors of $A_n$, we will now show that $\cD_n(\varepsilon)$ is indeed an $\sqrt{n}\vep$-net of $\{A_n{\bm v}: {\bm v} \in [0,1]^n\}$ having a sub-exponential size.  Since \eqref{eq:chatterjee} is immediate from \eqref{eq:sub_expo}, it only remains to show that $\cD_n(\varepsilon)$ is a $\sqrt{n}\varepsilon$-net. 

To this end, fix ${\bm v}\in [0,1]^n$, and expand ${\bm v}$ in the basis $\{{\bm p}_1,{\bm p}_2,\cdots,{\bm p}_n\}$ as
$${\bm v}=\sum_{i=1}^n\alpha_i{\bm p}_i,$$
where $\alpha_1,\alpha_2,\cdots,\alpha_n\in \R$ satisfies

\begin{align}\label{eq:bound_2}
\sum_{i=1}^n\alpha_i^2=\sum_{i=1}^n v_i^2\le n,
\end{align}
For $0\le k\le \ell$ setting $s_k:=\sqrt{\sum_{i\in I_k}\alpha_i^2}$, we will now find a vector ${\bm c}\in \cS_n(\varepsilon)$  such that
$$\norm{A{\bm v}-\sum_{i\in I}\lambda_i(A_n)c_i{\bm p}_i}_2\le \sqrt{n}\varepsilon.$$
If $I_k\neq \phi$ for some $k$, setting $j_k:=\max(0,\lceil {\log_2 s_k}\rceil)$ we note that 
 $(\alpha_i,i\in I_k)\in \cB_{|I_k|}(2^{j_k})$, and so there exists $(c_i,i\in I_k)\in \cC_{k}(j_k)$ such that
\begin{align}\label{eq:bound_3}
	\sum_{i\in I_k}(\alpha_i-c_i)^2\le \frac{|I_k|\varepsilon^2}{2^{2k+2}}.
	\end{align}
By our choice of $j_k$  we have $2^{j_k}\le 2s_k$ if $s_k\ge 1$, and $j_k=0$ if $s_k<1$. This gives
 $$\sum_{k=0}^\ell 2^{2j_k}=\sum_{k:j_k=0}2^{2j_k}+\sum_{k:j_k\ge 1}2^{2j_k}\le \ell+4\sum_{k=0}^\ell s_k^2\le \ell+4 \sum_{i=1}^n\alpha_i^2\le 5n,$$
where the last step uses \eqref{eq:bound_2}. Thus we have shown that
 ${\bm c}=(c_i)_{i \in I}\in \cS_n(\varepsilon)$.  
  Finally, recalling that $|\lambda_i(A_n)| \le 2^k$ for any $i \in I_k$, we note
 \begin{align*}
 	\norm{A{\bm v}-\sum_{i\in I}\lambda_i(A_n)c_i{\bm p}_i}_2^2=&\sum_{k=0}^\ell\sum_{i\in I_k}\lambda_i(A_n)^2(\alpha_i-c_i)^2+\sum_{i\notin I}\lambda_i(A_n)^2\alpha_i^2\\
	\le &\frac{\varepsilon^2}{4}\sum_{k=0}^\ell 2^{2k}\frac{|I_k|}{2^{2k}}+\frac{\varepsilon^2}{4}\sum_{i=1}^n\alpha_i^2\\
	\le &(N_n+n)\frac{\varepsilon^2}{4}\le \frac{n\varepsilon^2}{2},
 	\end{align*}
 where the first two inequalities follow by an use of \eqref{eq:bound_3} and \eqref{eq:bound_2}, respectively.  Thus we have shown that $\cD_n(\varepsilon)$ is indeed an $\sqrt{n}\varepsilon$-net.
 
 Therefore to complete the proof it suffices to show \eqref{eq:sub_expo}. To this effect, fix any $j_0,j_1,j_2,\cdots,j_\ell$ such that 
 $\sum_{k=0}^\ell 2^{2j_k}\le 5n$, and set $\cK=\cK(j_0,j_1,\ldots,j_\ell):=\{0\le k\le \ell:6\times 2^{k+j_k}\ge \varepsilon\sqrt{|I_k|}\}$. Thus we have
 \begin{align}\label{eq:bound_4}
 		\log |\cC_0(j_0)\times \cC_1(j_1)\times \cdots \times \cC_\ell(j_\ell)|\le \sum_{k\in \cK} |I_k|\log \Big(\frac{6}{\varepsilon}\frac{2^{k+j_k}}{\sqrt{|I_k|}}\Big) 
 		\end{align}
Further denote $N_n'=N_n'(j_0,j_1,\ldots,j_\ell):=\sum_{k\in \cK}|I_k|$, and obviously $N_n' \le N_n$. Now using Jensen's inequality, applied for $\log(\cdot)$, we have
 \begin{equation}\label{eq:net-bound-5}
 	\frac{1}{N_n'}\sum_{k\in \cK} |I_k|\log  \Big(\frac{6}{\varepsilon}\frac{2^{k+j_k}}{\sqrt{|I_k|}}\Big)
 	\le \log \Big\{\frac{6}{\varepsilon N_n'}\sum_{k\in \cK} 2^{k+j_k}\sqrt{|I_k|}\Big\}
 	\le \log \left\{\frac{6}{\varepsilon N_n'}\sqrt{\sum_{k=0}^\ell2^{2k}|I_k|}\sqrt{\sum_{k=0}^\ell  2^{2j_k}}\right\},
 	\end{equation}
where the last step follows by Cauchy-Schwarz's inequality. Since for any $k \ge 1$, and any $i \in I_k$, we have $|\lambda_i(A_n)| \ge 2^{k-1}$, we therefore deduce that
$$\sum_{k=0}^\ell 2^{2k}|I_k|\le |I_0| + 4\sum_{k=1}^\ell\sum_{i \in I_k}|\lambda_i(A_n)|^2 \le  N_n+4\sum_{i=1}^n\lambda_i(A_n)^2.$$
Now note that Assumption \ref{assu:mf} in particular implies that $\sum_{i=1}^n\lambda_i(A_n)^2\le Cn$ for some positive constant $C$. Therefore recalling that $\sum_{k=0}^\ell 2^{2j_k}\le 5n$, using \eqref{eq:bound_4}, and \eqref{eq:net-bound-5}, we deduce that 
\begin{align}
\log |\cC_0(j_0)\times \cC_1(j_1)\times \cdots \times \cC_\ell(j_\ell)| &\le N_n'\Big(\log \frac{6}{\varepsilon}\Big) +N_n'\log \frac{n\sqrt{5(4C+1)}}{N_n'}\notag\\
 &\le N_n\Big(\log \frac{6}{\varepsilon}\Big) +N_n\log \frac{n\sqrt{5(4C+1)}}{N_n},
\end{align}
where the last step uses the facts that $\lim_{n \ra \infty}\f{N_n}{n}=0$, and $x \mapsto x \log (1/x)$ is increasing near $0$. Therefore from the definition of the set $\cS_n(\vep)$ it now follows that
\[
|\cS_n(\varepsilon)|\le (1+\ell)^{1+\ell} \exp\left[N_n\Big(\log \frac{6}{\varepsilon}\Big) +N_n\log \frac{n\sqrt{5(4C+1)}}{N_n}\right].
\]
Now using the fact that $\lim_{n \ra \infty}\f{N_n}{n}=0$ again the proof completes. 
 \end{proof}

\begin{remark}
Note that the proof of Lemma \ref{lem:mean_field} goes through as long as the following hold:
	
	\begin{align}\label{eq:conj}\frac{1}{n}\sum_{i=1}^n\delta_{\lambda_i}(A_n)\stackrel{w}{\rightarrow}\delta_0,\quad \limsup\limits_{n\rightarrow\infty}\frac{1}{n}\sum_{i=1}^n\lambda_i(A_n)^2<\infty.
		\end{align}
For example, if $A_n$ is the adjacency matrix of the $n$-star graph $K_{1,n-1}$ then it does not satisfy the mean-field assumption. Indeed, this follows from observing that
$$\frac{1}{n}\sum_{i=1}^n\lambda_i^2(A_n)=\frac{2|E(K_{1,n-1})|}{n}\rightarrow 2.$$
However, all but $2$ of the eigenvalues of the adjacency matrix are zero. Therefore \eqref{eq:conj} holds here, and hence proof of Lemma \ref{lem:mean_field} goes through unchanged in this case. For the $n$-star graph, one can directly check that the mean-field approximation \eqref{eq:var_mean_field_prediction} is tight. 
In light of this and similar other examples, we believe the mean-field assumption can be weakened to \eqref{eq:conj}, and we conjecture that the conclusion of Theorem \ref{thm:main} continues to hold as long as \eqref{eq:conj} holds. 
	\end{remark}

		\section{Proof of Applications}\label{sec:Poa}

		\subsection{Proofs of Theorem \ref{thm:mean_field} and Theorem \ref{thm:ldp}} In this section we compute the limiting log partition function for asymptotically regular graphs. This is followed by the proof of large deviation principle for the empirical measure of the colors for such graphs.

		\begin{proof}[Proof of Theorem \ref{thm:mean_field}]
			(a) Since $A_n$ satisfies the mean-field assumption, applying Theorem \ref{thm:main} we get
			$$\lim_{n\rightarrow\infty}\frac{1}{n}[\Phi_n({\bm J},{\bm h})-\sup_{\gq\in \cP([q])^n}{\sf M}_n^{{\bm J},{\bm h}}(\gq)]=0.$$
			
			Proceeding to estimate ${\sf M}_n^{{\bm J},{\bm h}}(\gq)$, fixing $\delta>0$, we denote
\[
A^{(\delta)}_n(i,j):=A_n(i,j)1_{|\cR_n(i)-1|\le \delta}1_{|\cR_n(j)-1|\le \delta}.
\]			
Thus we have
			\begin{align}\label{eq:regular-1}
				&\frac{1}{n}\sum_{r=1}^q\sum_{i,j=1}^n A_n(i,j)\gq_i(r)\gq_j(r)\notag\\
				\le &\frac{1}{n}\sum_{r=1}^q\sum_{i,j=1}^n A_n^{(\delta)}(i,j)\gq_i(r)\gq_j(r)+\frac{1}{n}\sum_{i:|\cR_n(i)-1|> \delta}\sum_{r=1}^q\sum_{j=1}^n A_n(i,j)\gq_i(r)\gq_j(r)\notag\\
				& \qquad \qquad \qquad \qquad \qquad +\frac{1}{n}\sum_{j:|\cR_n(j)-1|> \delta}\sum_{r=1}^q\sum_{i=1}^n A_n(i,j)\gq_i(r)\gq_j(r)\notag\\
				= & \frac{1}{n}\sum_{r=1}^q\sum_{i,j=1}^n A_n^{(\delta)}(i,j)\gq_i(r)\gq_j(r)+\frac{2}{n}\sum_{i:|\cR_n(i)-1|> \delta}\sum_{r=1}^q\sum_{j=1}^n A_n(i,j)\gq_i(r)\gq_j(r).
			\end{align}
Note that the second term in the RHS of \eqref{eq:regular-1} is bounded above by  
			\begin{align}\label{eq:regular-2}
				\frac{2q}{n}\sum_{i:|\cR_n(i)-1|>\delta}\cR_n(i)
				\le \frac{2q}{n}\Big[\sum_{i=1}^n \cR_n(i)-(1-\delta)\sum_{i=1}^n1_{|\cR_n(i)-1|\le \delta}\Big]=2qa_n^{(\delta)},
			\end{align}	
			where
			\[
			a_n^{(\delta)}:=  \frac{1}{n}\Big[\sum_{i=1}^n \cR_n(i)-(1-\delta)\sum_{i=1}^n1_{|\cR_n(i)-1|\le \delta}\Big].
			\]			
			Considering the first term in the RHS of \eqref{eq:regular-1}, and noting that $A_n^{(\delta)}$ is a symmetric entries with non negative matrix whose row sums are bounded by $1+\delta$, we apply Gershgorin circle theorem to obtain
			\begin{align}
				\frac{1}{n}\sum_{r=1}^q\sum_{i,j=1}^n A^{(\delta)}_n(i,j)\gq_i(r)\gq_j(r)				\le\frac{1+\delta}{n} \sum_{r=1}^q \sum_{i=1}^n\gq_i(r)^2\label{eq:regular-3}.
			\end{align}
Combining \eqref{eq:regular-2}-\eqref{eq:regular-3}, along with the expression for ${\sf M}_n^{{\bm J},{\bm h}}(\gq)$, we get
			
			\begin{align}\label{eq:regular-4}
				\sup_{\gq\in \cP([q])^n}\frac{1}{n}{\sf M}_n^{{\bm J},{\bm h}}(\gq)
				\le &\sup_{\gq\in \cP([q])}\Big\{\frac{\beta}{2}\sum_{r=1}^q\gq(r)^2+\sum_{r=1}^qh_r\gq(r)-\sum_{r=1}^q\gq(r)\log \gq(r)\Big\}+\frac{q\beta }{2}\Big(\delta+2a_n^{(\delta)}\Big).
			\end{align}
			Now note that $a_n^{(\delta)} \ra \delta$ as $n\rightarrow\infty$ by \eqref{eq:upper_mean_field}-\eqref{eq:lower_mean_field}. Thus taking limits as $n\rightarrow\infty$ on both sides of \eqref{eq:regular-4}, we get
			$$\limsup\limits_{n\rightarrow\infty}\sup_{\gq\in \cP([q])^n}\frac{1}{n}{\sf M}_n^{{\bm J},{\bm h}}(\gq)\le \sup_{\gq\in \cP([q])}\Big\{\frac{\beta}{2}\sum_{r=1}^q\gq(r)^2+\sum_{r=1}^qh_r\gq(r)-\sum_{r=1}^q\gq(r)\log \gq(r)\Big\}+2q\beta \delta,$$ from which the upper bound of 
			\eqref{eq:strong_meanfield_conclusion} follows, as $\delta>0$ is arbitrary.
			
			\medskip			
			
			\noindent			
			For the lower bound, taking a supremum over all $\gq=\prod_{i=1}^n\gq_i$ such that $\gq_i$ is same for all $i$, we have
			\begin{align*}	\sup_{\gq\in \cP([q])^n}\frac{1}{n}{\sf M}_n^{{\bm J},{\bm h}}(\gq)\ge 	\sup_{\gq\in \cP([q])}\Big\{\frac{\beta}{2}\sum_{r=1}^q\gq(r)^2\frac{1}{n}\sum_{i=1}^n\cR_n(i)+\sum_{r=1}^qh_r\gq(r)-\sum_{r=1}^q\gq(r)\log \gq(r)\Big\},
			\end{align*}
			which on dividing by $n$, and taking limits using \eqref{eq:lower_mean_field}, gives the lower bound in \eqref{eq:strong_meanfield_conclusion}. 
			This completes the proof of part (a).

			\bigskip
			
			\noindent
(b) To prove part (i), we note that $\cR_n(i)=1$ for all $i\in [n]$, and thus both \eqref{eq:upper_mean_field} and \eqref{eq:lower_mean_field} hold trivially.

			\noindent			
			Turning to prove part (ii), note that $\cR_n(i)=\frac{d_i(\Graph_n)}{np_n}$ with $d_i(\Graph_n)$ denoting the degree of vertex $i\in [n]$. Since the number of edges $|E_n|$ has a $\dBin({n\choose 2},p_n)$ distribution, 
			$$\frac{1}{n}\sum_{i=1}^n\cR_n(i)=\frac{2|E_n|}{n^2p_n}{\rightarrow}1, \text{ in probability}.$$
			This  verifies \eqref{eq:lower_mean_field}. To check \eqref{eq:upper_mean_field}, fixing $\delta >0$, it suffices to check that $\lim_{n\rightarrow\infty}\frac{1}{n}\E N_n^{(\delta)}=0$, where 
			$$N_n^{(\delta)}:=\sum_{i=1}^n1_{|d_i(\Graph_n)-np_n|>np_n\delta}.$$ This follows using Chebyshev's inequality:
			$$\frac{1}{n}\E N_n^{(\delta)}=\f{1}{n}\sum_{i \in [n]}\P(|d_i(\Graph_n)-np_n|>np_n\delta)\le \frac{(n-1)p_n(1-p_n)}{n^2p_n^2\delta^2}\rightarrow0,$$
			as $np_n\rightarrow\infty$.

		\end{proof}
\medskip

\noindent
Now as an application of Theorem \ref{thm:mean_field}, we derive the following large deviation principle. As a byproduct we also get an exponential concentration of the average sample spins in Ising model.		
		
			\begin{proof}[Proof of Theorem \ref{thm:ldp}]
(a) The proof of this theorem is based on Baldi's theorem (cf.~\cite[Theorem 4.5.20]{DZ}). To this end, we first need to compute logarithmic moment generating function. Fixing a vector ${\bm {\bm t}}=(t_1,t_2,\ldots,t_q)\in\R^q$, using Theorem \ref{thm:mean_field}, one has 
					\begin{align*}
						\frac{1}{n}\log \E_{\mu_n} e^{{n}\sum_{r\in [q]} t_r L_n(r)}=&\frac{1}{n}\left[\Phi_n({\bm J},{\bm h}+{\bm t})-\Phi_n({\bm J},{\bm h})\right]\\
						\stackrel{n\rightarrow\infty}{\longrightarrow}&\sup_{\gq\in \cP([q])}\Big[\frac{\beta}{2}\sum_{r=1}^q\gq(r)^2-\sum_{r=1}^q \gq(r)\log \gq(r)+\sum_{r=1}^q (h_r+t_r)\gq(r)\Big]\\-&\sup_{\gq\in \cP([q])}\Big[\frac{\beta}{2}\sum_{r=1}^q\gq(r)^2-\sum_{r=1}^q \gq(r)\log \gq(r)+\sum_{r=1}^q h_r\gq(r)\Big].
					\end{align*} 
Denoting the RHS above by $\Lambda({\bm t})$ we note that
$$\Lambda({\bm t})=\sup_{\mu\in \cP([q])}\Big\{\sum_{r\in [q]}t_r\mu_r-\widetilde{I}_{\beta,{\bm h}}(\mu)\Big\}.$$
Therefore, applying the duality lemma (see \cite[ Lemma 4.5.8]{DZ}),  we have 
$$\widetilde{I}_{\beta,{\bm h}}(\mu)=\sup_{{\bm t}\in \R^q}\Big\{\sum_{r\in [q]}t_r\mu_r-\Lambda({\bm t})\Big\}.$$
Next note that the set $\cP([q])$ being compact, the law of $L_n(\cdot)$ is automatically exponentially tight. Thus using the fact that $\Lambda({\bm t}) < \infty$ for all ${\bm t} \in \R^q$, applying \cite[Theorem 4.5.20(a)]{DZ} we obtain that for any closed set $F\subset \cP([q])$,
\[
\limsup_{n \ra \infty} \f{1}{n}\log \mu_n(L_n \in F)  \le - \inf_{\mu \in F} \wt{I}_{\be, {\bm h}}(\mu).
\] 
To derive the lower bound we use part (b) of \cite[Theorem 4.5.20]{DZ}. To this end, we note that it is enough to prove that any $\mu\in \cP([q])$ is an exposed point of $\widetilde{I}_{\beta,{\bm h}}(.)$, i.e.~for any $\nu\in \cP([q])$ with $\mu\ne \nu$ there exists ${\bm t}\in\R^q$ such that
					$$\sum_{r\in [q]}\Big\{\frac{\beta}{2}\sum_{r=1}^q\mu_r^2-\sum_{r=1}^q \mu_r\log \mu_r+\sum_{r=1}^q (h_r+t_r)\mu_r\Big\}>\sum_{r\in [q]}\Big\{\frac{\beta}{2}\sum_{r=1}^q\nu_r^2-\sum_{r=1}^q \nu_r\log \nu_r+\sum_{r=1}^q (h_r+t_r)\nu_r\Big\}.$$
This follows on noting the existence of $r\in [q]$ such that $\mu_r>\nu_r$, and then choosing $t_r$ large enough for all r such that $\mu_r >\nu_r$, and $t_r=0$ for all $r$ such that $\mu_r \le \nu_r$. 

\noindent
Now to prove \eqref{eq:conc}, we note that the function  $\mu\mapsto I_{\beta,{\bm h}}(\mu)$ is a non constant analytic function on a compact set, an thus the infimum is attained on a finite set $K_{\be, {\bm h}}$. Thus \eqref{eq:conc} follows from the large deviation principle on noting that the set $\{\nu\in \cP([q]):\min_{\mu \in K_{\beta,{\bm h}}}\norm{\nu-\mu}_\infty\ge \delta\}$ is closed.

\smallskip

\noindent
(b) By the last conclusion of part (a) it suffices to minimize the function $I_{\beta,{\bm h}}(\mu)$. 
					To begin introduce the variable $m=\mu_1-\mu_2\in [-1,1]$ and note that 
					\begin{align*}
				I_{\beta,{\bm h}}(\mu)=	-\frac{\be}{4}m^2-\frac{B}{2}m+H(m)-\Big[\frac{\beta}{4}-\frac{h_1+h_2}{2}\Big].
					\end{align*}
					The optimization of this function has been carried out in \cite[Section 1.1.3]{dm}, where it is shown that
optimum is at $m=0$ for $\beta\le 2,B=0$, at $m=\pm m_{\beta/2,0}$ for $\beta>2,B=0$, and at $m=m_{\beta/2,B/2}$ for $\beta>0,B\ne 0$. This, along with the symmetry of the Ising model for $B=0$ completes the proof of part (b).

			\end{proof}

		\subsection{Proofs of Theorem \ref{thm:bipartite} and Theorem  \ref{thm:lp}} In this section we prove the convergence of log partition function for bi-regular bipartite graphs, followed by the same for a sequence of graphs converging in cut metric.
					
		\begin{proof}[Proof of Theorem \ref{thm:bipartite}]
		To begin first note that \eqref{eq:alpha}, along with $a_nc_n=(n-a_n)d_n$ implies $c_n=\Theta(d_n),$ and therefore we deduce that $c_n$ and $d_n$ individually converge to $\infty$, as $n \ra \infty$. This further implies that {$|E_n|=a_nc_n\gg n$}. Thus Corollary \ref{cor:main} is applicable, and it suffices only to consider the asymptotics of 
		$\sup_{\gq\in \cP([q])^n}\mathsf{M}_n^{\beta,B}(\gq).$
		For computing the supremum in this setting, denoting $b_n:=n-a_n$ we have
		\begin{align*}
		\mathsf{M}_n^{\be,B}(\gq)&=\be\sum_{i\in [a_n]j\in [b_n],r\in [2]}\gq^{(1)}_i(r)\gq^{(2)}_j(r)A_n(i,j)-\sum_{i\in [a_n],r\in [2]}\gq^{(1)}_i(r)\log \gq^{(1)}_i(r)\\
		& \qquad \qquad \qquad \qquad \qquad\qquad \qquad \qquad \qquad \qquad -\sum_{j\in [b_n],r\in [2]}\gq^{(2)}_j(r)\log \gq^{(2)}_j(r).
		\end{align*}
		Introducing variables {$\gs^{(1)}_i:=\gq_i^{(1)}(1)-\gq_i^{(1)}(2)$, and $\gs^{(2)}_j:=\gq_j^{(2)}(1)-\gq_j^{(2)}(2)$, and noting that $\sum_{k \in [2]}\gq_i^{(1)}(k)= \sum_{k \in [2]}\gq_j^{(2)}(k) =1 $}, the RHS above becomes
		\begin{align}
		=&\frac{\be}{2}\sum_{i\in [a_n],j\in [b_n]}(1+\gs^{(1)}_i\gs^{(2)}_j)A_n(i,j)+\sum_{i\in [a_n]}H(\gs^{(1)}_i)+\sum_{j\in [b_n]}H( \gs^{(2)}_j)\notag\\
		=&\frac{\be}{2}\sum_{i\in [a_n],j\in [b_n]}\gs^{(1)}_i \gs^{(2)}_jA_n(i,j)+\sum_{i\in [a_n]}H(\gs^{(1)}_i)+\sum_{j\in [b_n]}H( \gs^{(2)}_j)+\frac{\be}{2}\frac{a_nc_n}{c_n+d_n}.\label{eq:new_func}
\end{align}		
Hence, it suffices to maximize \eqref{eq:new_func}
		over  the set $\{\gs_i^{(1)}\in [-1,1] ,i\in [a_n];\gs_j^{(2)}\in [-1,1],j\in [b_n]\}$.
		
		Fixing $n$ first note that the optimum occurs at an interior point where $\gs_i^{(1)}\in (-1,1),\gs_j^{(2)}\in (-1,1)$, for any $i\in [a_n],j\in [b_n]$. This is due to the facts that for any $i\in [a_n]$, we have $${\frac{\partial }{\partial \gs_i^{(1)}}\mathsf{M}_n^{\be,B}\Big|_{\gs_i^{(1)}\rightarrow -1+}=+\infty,\quad \frac{\partial }{\partial \gs_i^{(1)}}\mathsf{M}_n^{\be,B}\Big|_{\gs_i^{(1)}\rightarrow 1-}=-\infty,}$$ and a similar argument holds for $\gs_j^{(2)}$ for $j\in [b_n],$ as well. Thus differentiating with respect to $\gs_i^{(1)}, \gs_j^{(2)}$ and equating to $0$, any  optimum  satisfies the following equations
		\begin{align}\label{eq:m1}
		\gs_i^{(1)}=\tanh\Big(\beta \sum_{j\in [b_n]}A_n(i,j)\gs_j^{(2)}\Big),\\
		\label{eq:m2}
		\gs_j^{(2)}=\tanh\Big(\beta \sum_{i\in [a_n]}A_n(i,j)\gs_i^{(1)}\Big).
		\end{align}
{We now split the proof into four different cases.

\medskip

\noindent
{\bf Case 1:} $\be >0$, and $\be^2 p(1-p) <1$. 		

\noindent
Since $\be >0$, and $H(x)=H(-x)$, without loss of generality we can assume that for any optimum we have $\gs_{i}^{(1)},\gs_{j}^{(2)}\ge 0$. Next combining \eqref{eq:m1}, and \eqref{eq:m2}, for every $i \in [a_n]$ we get 
		\begin{align}\label{eq:m3}
		\gs_i^{(1)}=\tanh\Big(\beta \sum_{j\in [b_n]}A_n(i,j)\tanh\Big(\beta\sum_{k\in [a_n]}A_n(j,k)\gs_k^{(1)}\Big)\Big).
		\end{align}
Letting $\gs_{i_0}^{(1)}:=\arg\max_{i\in [a_n]}\gs_{i}^{(1)}$, \eqref{eq:m3} further yields
		\begin{align}\label{eq:m4}
		\gs_{i_0}^{(1)}=&\tanh\Big(\beta \sum_{j\in [b_n]}A_n(i,j)\tanh\Big(\beta\sum_{k\in [a_n]}A_n(j,k)\gs_k^{(1)}\Big)\Big)\notag\\
		\le &\tanh\Big(\beta \sum_{j\in [b_n]}A_n(i,j)\tanh\Big(\beta\sum_{k\in [a_n]}A_n(j,k)\gs_{i_0}^{(1)}\Big)\Big)\notag\\
		=&\tanh\Big(\beta \frac{c_n}{c_n+d_n}\tanh\Big(\beta\gs_{i_0}^{(1)}\frac{d_n}{c_n+d_n}\Big)\Big)=:\eta_{\beta,\frac{d_n}{c_n+d_n}}(\gs_{i_0}^{(1)}).
		\end{align}		
It is easy to note that 		
$$\norm{\f{d\eta_{\beta,\frac{d_n}{c_n+d_n}}(\gs)}{d \gs}}_\infty=\beta^2\frac{c_nd_n}{(c_n+d_n)^2}\stackrel{n\rightarrow\infty}{\longrightarrow}\beta^2p(1-p)<1.$$
Thus $\gs \mapsto \eta_{\be,\f{d_n}{c_n+d_n}}(\gs)$ is a contraction. This implies that for any $\gs >0$ 
\[
\eta_{\be,\f{d_n}{c_n+d_n}}(\gs)= \left| \eta_{\be,\f{d_n}{c_n+d_n}}(\gs) - \eta_{\be,\f{d_n}{c_n+d_n}}(0)\right| < |\gs - 0| = \gs,
\]
for all large $n$. Using \eqref{eq:m4}, for large $n$, we therefore deduce that $\gs_{i_0}^{(1)}$ must be equal to zero. This further implies that $\gs_i^{(1)}$ must be equal to zero for all $i \in [a_n]$. Similar arguments hold for $\gs_j^{(2)}$, proving $\gs_j^{(2)}=0$ for all $j \in [b_n]$. Plugging in the values of $\gs_i^{(1)}$, and $\gs_j^{2)}$ in the RHS of \eqref{eq:new_func} we have
		$$\sup_{\gq\in \cP([q])^n}\mathsf{M}_n^{\be,0}(\gq)=\frac{\beta}{2} \frac{a_nc_n}{c_n+d_n}+n\log 2,$$ which on dividing by $n$ and taking limits proves {\bf Case 1}.

\medskip

\noindent
{\bf Case 2:} $\be <0$, and $\be^2 p(1-p) <1$.	

\noindent
Note that one can rewrite $\mathsf{M}_n^{\beta,0}(\gq)$ as 
		$$\frac{(-\be)}{2}\sum_{i\in [a_n],j\in [b_n]}\gs^{(1)}_i(-\gs^{(2)}_j)A_n(i,j)+\sum_{i\in [a_n]}H(\gs^{(1)}_i)+\sum_{j\in [b_n]}H( -\gs^{(2)}_j)+\frac{\be}{2}\frac{a_nc_n}{c_n+d_n} .$$
Since $\beta<0$, one can argue that for any optimum we must have $\gs_i^{(1)}$ and $-\gs_j^{(2)}$ non negative for all $i \in [a_n]$, and $j \in [b_n]$.  The rest of the arguments is similar to {\bf Case 1}. We omit the details. 

\medskip

\noindent
{\bf Case 3:} $\be >0$, and $\be^2 p(1-p) >1$.		
	
\noindent	
We begin by noting that 
\[
\f{d\eta_{\beta, \frac{d_n}{c_n+d_n}}}{d \gs}(\gs)= \beta^2\frac{c_nd_n}{(c_n+d_n)^2} \sech^2\left( \tanh\Big(\beta\gs\frac{d_n}{c_n+d_n}\Big)\right)\sech^2\Big(\beta\gs\frac{d_n}{c_n+d_n}\Big),
\]	
which is decreasing in $\gs$, and goes to zero as $\gs \ra \infty$. Further noting that $$\f{d\eta_{\beta, \frac{d_n}{c_n+d_n}}}{d \gs}(\gs)\Big|_{\gs=0}=\beta^2\frac{c_nd_n}{(c_n+d_n)^2}\stackrel{n\rightarrow\infty}{\longrightarrow}\beta^2 p(1-p)>1,$$
we deduce that there is a unique positive root of the equation $\gs=\eta_{\beta,\frac{d_n}{c_n+d_n}}(\gs)$, denoted by  $\gs_{\beta,\frac{d_n}{c_n+d_n}}$, for all $n$ large enough. Also \eqref{eq:m4} implies $$\max_{i\in [a_n]}\gs_i^{(1)}=\gs_{i_0}^{(1)}\le \gs_{\beta,\frac{d_n}{c_n+d_n}}.$$ By a similar argument we also deduce $$\min_{i\in [a_n]}\gs_i^{(1)}\ge  \gs_{\beta,\frac{d_n}{c_n+d_n}},$$
and so $\gs_i^{(1)}=\gs_{\beta,\frac{d_n}{c_n+d_n}}$ for all $i\in [a_n]$. Plugging in this solution in \eqref{eq:m2}  gives $\gs_j^{(2)} =\gs_{\beta,\frac{c_n}{c_n+d_n}}$ for all $j\in [b_n]$. Thus the optimum solution is 
$$\gs_i^{(1)}=\gs_{\beta,\frac{d_n}{c_n+d_n}},\text{ for all }i\in [a_n], \quad \gs_j^{(2)} =\gs_{\beta,\frac{c_n}{c_n+d_n}}\text{ for all }j\in [b_n].$$
Plugging in this optimal solution in the RHS of \eqref{eq:new_func} gives
$$\sup_{\gq\in \cP([q])^n}\mathsf{M}_n^{\beta,0}(\gq)= \frac{\beta a_nc_n}{2(c_n+d_n)}\Big\{ 1+\gs_{\beta,\frac{d_n}{c_n+d_n}}\gs_{\beta,\frac{c_n}{c_n+d_n}})\Big\}+a_n H\left(\gs_{\beta,\frac{d_n}{c_n+d_n}}\right)+b_n H\left(\gs_{\beta,\frac{d_n}{c_n+d_n}}\right), $$
from which part (b) follows on dividing by $n$ and taking limits, on noting that the function $p\mapsto \gs_{\beta,p}$ is continuous.

\medskip

\noindent
{\bf Case 4:} $\be <0$, and $\be^2 p(1-p) >1$.

\noindent
This can be done by combining the arguments of {\bf Case 2}, and {\bf Case 3.} We omit the details.
		
		
\bigskip

\noindent
Note that the above four cases complete the proof, barring the convergence of $\Phi_n(\be,0)$ at $\be=\pm \be_c(p)=\pm \sqrt{p(1-p)}$. To complete the proof first we use the fact that $|\tanh(x)| < |x|$, for any $x \ne 0$, and deduce that $\gs_{\be,p} \ra 0$, as $\be \ra \pm \be_c$. This implies that $$\Phi	(\be,0):=\frac{\beta p(1-p)}{2}+\frac{|\beta| p(1-p)}{2}\gs_{|\beta|,p}\gs_{|\beta|,1-p}+p H(\gs_{|\beta|,p})+(1-p) H(\gs_{|\beta|,1-p})$$ is continuous for all $\be$. 
Since $\{\Phi_n(\cdot,0)\}$ are convex functions, and limit of such functions is also a convex function, using the fact that $\limsup_{n \ra \infty} \f{1}{n}\Phi_n(\be,0) < \infty$ at $\be=\pm \be_c(p)$, the proof completes by a standard analysis argument.}		
		\end{proof}
{		
\begin{remark}
Even though we do not pursue it here, by combining the arguments of Theorem \ref{thm:mean_field} and Theorem \ref{thm:bipartite} one should be able to prove Theorem \ref{thm:bipartite}  for a sequence {\em asymptotically bi-regular bipartite graphs}.  We believe a similar universality result for the limiting log partition function for the $q$ Potts model holds for general $q$-partite graphs as well, though proving it will require an analysis of fixed points in $q$ dimensional equations for $q>2$. 
\end{remark}		
}	

\medskip

Finally we prove Theorem \ref{thm:lp}.		
				{
				\begin{proof}[Proof of {Theorem} \ref{thm:lp}]
			
			By assumption $W_{nA_n}$
			 converges to  $W$ in the cut metric, and therefore by \cite[Proposition C5 and Proposition C15]{sparse1}, we have $\lim_{n\rightarrow\infty}\frac{|E_n|}{n}=\infty.$
			Thus applying Corollary \ref{cor:main}, we note that it suffices to show
			\begin{align}\label{eq:graph_limit_1}
			\lim_{n\rightarrow\infty}\frac{1}{n}\sup_{\gq\in\mathcal(\mathcal{P}[q])^n}{\sf M}_n^{{\bm J},{\bm h}}(\gq)=\sup_{{\bm \rho}\in {{\bf FP}_q}}F^{{\bm J},{\bm h}}(W,{\bm \rho}).
			\end{align}
To this end, setting $\rho_r(x)=\gq_i(r)$ for $(\frac{i-1}{n},\frac{i}{n}]$ for each $1\le r\le q,1\le i\le n$, we note that
\beq\label{eq:connect_M_F}
\frac{1}{n}{\sf M}_n^{{\bm J},{\bm h}}(\gq)=F^{{\bm J},{\bm h}}(W_{nA_n},{\bm \rho}).
\eeq
Since $nW_{A_n}$ converges to $W$ in the cut metric we have
\beq\label{eq:cut_norm_implication}
\sup_{{\bm \rho}\in {\bf FP}_q}|F^{{\bm J}, {\bm h}}(W_{nA_n},{\bm \rho})-F^{{\bm J}, {\bm h}}(W,{\bm \rho})|\rightarrow 0.
\eeq
This implies that 
\begin{align*}
\limsup_{n \ra \infty} \frac{1}{n}\sup_{\gq\in\mathcal(\mathcal{P}[q])^n}{\sf M}_n^{{\bm J},{\bm h}}(\gq) \le \sup_{{\bm \rho}\in {{\bf FP}_q}}F^{{\bm J},{\bm h}}(W,{\bm \rho}).
\end{align*}
Thus to establish \eqref{eq:graph_limit_1}, we need to prove the other side of the inequality. Turning to prove the same, we note that it suffices to show that given any ${\bm \rho}\in {\bf FP}_q$ there exists ${\bm \rho}^{(n)}\in{\bf FP}_q$, with $\rho_r$ being constant on $(\frac{i-1}{n},\frac{i}{n}]$ for $1 \le i \le n, 1 \le r \le q$, such that 
		   \begin{align}\label{eq:graph_limit_2}
		   	\lim_{n\rightarrow\infty}F^{{\bm J}, {\bm h}}(W,{\bm \rho}^{(n)})=F^{{\bm J}, {\bm h}}(W,{\bm \rho}).
		   	\end{align}
Indeed, using \eqref{eq:connect_M_F}, we deduce that, for any ${\bm \rho} \in {{\bf FP}}_q$
\[
F^{{\bm J}, {\bm h}}(W,{\bm \rho}) \le \left|F^{{\bm J}, {\bm h}}(W,{\bm \rho}) - F^{{\bm J}, {\bm h}}(W_{nA_n},{\bm \rho}^{(n)})\right|+ \frac{1}{n}{\sf M}_n^{{\bm J},{\bm h}}(\gq).
\]
Next taking a supremum over $\gq \in \cP([q])^n$, followed by a liminf on the both sides, and using \eqref{eq:cut_norm_implication}, and \eqref{eq:graph_limit_2}, we further obtain that
\[
F^{{\bm J}, {\bm h}}(W,{\bm \rho}) \le \liminf_{n \ra \infty} \frac{1}{n}\sup_{\gq\in \cP([q])^n}{\sf M}_n^{{\bm J},{\bm h}}(\gq).
\]
Next taking another supremum over ${\bm \rho}\in {\bf FP}_q$, we complete the proof of \eqref{eq:graph_limit_1}.

\medskip

\noindent
Now it only remains to establish \eqref{eq:graph_limit_2}. A standard measure theoretic arguments {yields} the existence $\rho^{(n)}\in {\bf FP}_{q}$, with $\rho_r$ being constant on $(\frac{i-1}{n},\frac{i}{n}]$ for $1 \le i \le n, 1 \le r \le q$, such that 
		  \begin{align}\label{eq:est_as}\lim_{n\rightarrow\infty}\max_{r=1}^q|\rho^{(n)}_r(x)-\rho_r(x)|= 0, \text{ Lebesgue almost surely}.
		 \end{align}
Therefore, noting $\norm{W}_1 <\infty$, using dominated convergence theorem, and the fact that the function $x\mapsto x\log x$ is continuous on $[0,1]$ we prove \eqref{eq:graph_limit_2}. This completes the proof of the theorem.  
			\end{proof}
}


\begin{thebibliography}{99}
\bibitem{structure_learning}A.~Anandkumar, V.~Tan, F.~Huang, and A.~Willsky.
\newblock High-dimensional structure estimation in Ising models: Local separation criterion. 
\newblock {\it The Annals of Statistics}, Vol.~40(3), 1346--1375, 2012.
				
				
				
\bibitem{bgt}M.~Bayati, D.~Gamarnik, and P.~Tetali.
\newblock Combinatorial approach to the interpolation method and scaling limits in sparse random graphs.
\newblock{\em The Annals of Probability},	Vol.~41(6), 4080--4115, 2013.
				
				
				
				
\bibitem{Bai}
Z.~D.~Bai.
\newblock Methodologies in spectral analysis of large-dimensional random
matrices, a review (with discussion). {\em Statist.~Sinica}, Vol.~9(3), 
611--662, 1999.

\bibitem{Ben-Mon}J.~Bento and A.~Montanari. 
			\newblock Which graphical models are difficult to learn? 
\newblock{\em Neural Information Processing Systems}, Vol.~22, 1303--1311, 2009.

\bibitem{BM}B.~Bhattacharya and S.~Mukherjee.
\newblock Inference in Ising Model. {\em Preprint}. Available at \url{http://arxiv.org/abs/1507.07055}, 2015.	

\bibitem{BC}M.~Biskup and L.~Chayes.
\newblock Rigorous Analysis of Discontinuous Phase Transitions via Mean-Field Bounds.
\newblock {\em Communications in Mathematical Physics}, Vol.~238(1), 53--93, 2003.		
		
		\bibitem{BGRW}
		P.~Blanchard, D.~Gandolfo, J.~Ruiz, and M.~Wouts.
\newblock Thermodynamic vs topological phase transition:  Cusp
in the Kert\'ez line.
\newblock{\em Europhysics Letters},
Vol.~82(5), 50003, 2008.

	

\bibitem{sparse1}
C.~Borgs, J.~Chayes, H.~Cohn, and Y.~Zhao.
\newblock An $L_p$ theory of sparse convergence I: limits, sparse random graph models, and power law distributions.
\newblock {\em Preprint}. Available at \url{http://yufeizhao.com/papers/2014-LpLimit1.pdf}.

\bibitem{sparse2}
C.~Borgs, J.~Chayes, H.~Cohn, and Y.~Zhao.
\newblock An $L_p$ 	theory of sparse convergence II: LD convergence, quotients, and right convergence.
\newblock {\em Preprint}. Available at \url{http://yufeizhao.com/papers/2014-LpLimit2.pdf}.

\bibitem{graph_limits_I}
C.~Borgs, J.~Chayes, L.~Lov\'asz, V.~T.~S\'os, and K.~Vesztergombi. Convergent sequences of dense graphs I: Subgraph frequencies, metric properties and testing. {\it Advances in Mathematics}. Vol.~219, 1801--1851, 2009.

\bibitem{graph_limits_II}
C. Borgs, J.~Chayes, L.~Lov\'asz, V.~T.~S\'os, and K.~Vesztergombi, Convergent sequences of dense graphs II. Multiway cuts and statistical physics. {\it Annals of Mathematics}, Vol.~176, 151--219, 2012.

\bibitem{bresler}
G.~Bresler.
\newblock Efficiently learning Ising models on arbitrary graphs. {\it Symposium on Theory of Computing (STOC)}, 771--782, 2015.

\bibitem{CP}A.~Cant and P.~A.~Pearce.
\newblock Mean-field limits of the quantum Potts model.
\newblock {\em Communications in Mathematical Physics}, Vol.~90(3), 373--383, 1983. 

\bibitem{chandler}D.~Chandler.
\newblock {\em Introduction to Modern Statistical Mechanics.}  Oxford: Oxford University Press, 1987.

\bibitem{Chatterjee} 
S.~Chatterjee.
\newblock Estimation in spin glasses: A first step.  {\it The Annals of Statistics}, Vol.~35(5), 1931--1946, 2007. 

\bibitem{chatterjee_dembo}
S.~Chatterjee and A.~Dembo.
\newblock Nonlinear large deviations.
\newblock {\em Advances in Mathematics,} to appear. Available at \url{http://arxiv.org/pdf/1401.3495v5.pdf}.

\bibitem{CL}
A.~Coja-Oghlan  and A.~ Lanka. 
\newblock The spectral gap of random graphs with given expected degrees. 
\newblock {\em The electronic journal of combinatorics}, Vol.~16(1), R138, 2009.


\bibitem{cet}
M.~Costeniuc, R.~S.~Ellis, and H.~Touchette.
\newblock Complete Analysis of Phase Transitions and Ensemble Equivalence for the Curie-Weiss-Potts Model.
\newblock{\em Journal of Mathematical Physics}, Vol.~46(6), 063301, 2005.

\bibitem{CLV}
F.~Chung, L.~Lu, and V.~Vu
\newblock The Spectra of Random Graphs with Given Expected Degrees.
\newblock{\em Journal of Internet Mathematics},
    Vol.~1(3), 257--275, 2003.

\bibitem{dmaop} A.~Dembo and A.~Montanari.
\newblock  Ising models on locally tree-like graphs. 
\newblock{\em The Annals of Applied Probability}, Vol.~20(2), 565--592, 2010.   

\bibitem{dm}
A.~Dembo and A.~Montanari.
\newblock Gibbs measures and phase transitions on sparse random graphs. 
\newblock {\em Brazilian Journal of Probability and Statistics}, Vol.~24(2), 137--211, 2010.

\bibitem{dmsen}
A.~Dembo, A.~Montanari, and S.~Sen.
\newblock Extremal cuts of sparse random graphs.
{\em The Annals of Probability}, to appear. Available at \url{http://arxiv.org/abs/1503.03923}.

\bibitem{dms}
A.~Dembo, A.~Montanari, and N.~Sun.
\newblock Factor models on locally tree-like graphs. 
\newblock {\em The Annals of Probability}, Vol.~41(6), 4162--4213, 2013.

\bibitem{dmss}
A.~Dembo, A.~Montanari, A.~Sly, and N.~Sun.
\newblock The replica symmetric solution for Potts models on d-regular graphs. 
\newblock {Communications in Mathematical Physics}, Vol.~327(2), 551--575, 2014.


\bibitem{DZ}
A.~Dembo and O.~Zeitouni.
\newblock Large Deviations Techniques and Applications.
{\em Springer}, Vol.~38, 2009.

\bibitem{DGH}S.~Dommers, C.~Giardin\`{a}, and R.~van der Hofstad.
\newblock Ising models on power-law random graphs. 
\newblock {\em Journal of Statistical Physics}, 141(4), 638--660, 2010.

\bibitem{em}
P.~Eichelsbacher and B.~Martschink.
\newblock{ On rates of convergence in the Curie-Weiss Potts model with an external field}.
\newblock{\em Annales de I'Insitut Henri Poincar\'e Probabilit\'es et Statistiques},
    Vol.~51(1), 252-282, 2015.
    
    
\bibitem{en}
R.~S.~Ellis and M.~E.~Newman.
\newblock The statistics of Curie-Weiss models. 
\newblock {\em Journal of Statistical Physics}, Vol.~19(2), 149--161, 1978.

\bibitem{ew}
R.~S.~Ellis and K.~Wang. 
\newblock Limit theorems for the empirical vector of the Curie-Weiss-Potts model. 
\newblock {\em Stochastic processes and their applications}, Vol.~35(1), 59--79, 1990.

\bibitem{ew_mle}
R.~S.~Ellis and K.~Wang.
\newblock Limit theorems for maximum likelihood estimators in the Curie-Weiss-Potts model.
\newblock {\em Stochastic processes and their applications}, Vol.~40(2), 251--288, 1992.


\bibitem{Gar}
D.~Gamarnik. 
\newblock Correlation decay method for decision, optimization, and inference in large-scale networks.
\newblock {\em Tutorials in Operations Research, INFORMS}, 2013.

\bibitem{FO}
U.~Feige and E.~Ofek.
\newblock Spectral techniques applied to sparse random graphs. \newblock{\em Random Structures \& Algorithms}, Vol.~27(2), 251--275, 2005. 

\bibitem{Hop}
J.~Hopfield.
\newblock{Neural networks and physical systems with emergent
collective computational abilities}. 
{\it Proceedings of the national academy of sciences}, Vol.~79(8), 2554--2558, 1982.

\bibitem{Ising}
E.~Ising.
\newblock Beitrag zur theorie der ferromagnetismus. {\it Zeitschrift f\"ur Physik}, Vol.~31(1), 253--258, 1925.

\bibitem{JS}M.~Jerrum and A.~Sinclair.
\newblock Polynomial-time approximation algorithms for the Ising model. 
\newblock {\em SIAM Journal on computing}, Vol.~22(5), 1087--1116, 1993.

\bibitem{ks}
M.~Krivelevich and B.~Sudakov.
\newblock The Largest Eigenvalue of Sparse Random Graphs.
\newblock {\em Combinatorics, Probability and Computing}, Vol.~12 (1), 61--72, 2003.

\bibitem{lovasz_book}
L.~Lov\'asz. {\it Large networks and graph limits}. Vol.~60, AMS, Providence, RI, 2012.


\bibitem{parisi} G.~Parisi. 
\newblock {\em Statistical Field Theory.} New York: Addison-Wesley, 1988.

\bibitem{MS}V.~D.~Milman and G.~Schechtman.
\newblock {\em Asymptotic theory of finite-dimensional normed spaces}. Lecture Notes in Mathematics, Vol.~1200. Springer-Verlag, Berlin, 1986.




\bibitem{N1}M.~Niss. 
\newblock History of the Lenz-Ising model 1920-1950: from ferromagnetic to cooperative phenomena. 
\newblock {\em Archive for history of exact sciences}, Vol.~59(3), 267--318, 2005.




\bibitem{Potts}
R.~Potts.
\newblock Some Generalized Order-Disorder Transformations. {\it  Mathematical Proceedings of the Cambridge philosophical society}, Vol.~48(1), 106--109, 1952.

\bibitem{highdim_ising}
P.~Ravikumar, M.J.~Wainwright, and J.~Lafferty.
\newblock High-dimensional Ising model selection using $\ell_1$-regularized logistic regression. {\it The Annals of Statistics}, Vol.~38 (3), 1287--1319, 2010.


\bibitem{SK}
D.~Sherrington and S.~Kirkpatrick.
\newblock{\em Solvable model of a spin-glass}.
Physical Review Letters, Vol.~35(26), 1792--1796, 1975.

\bibitem{Sly-Sun}
 A.~Sly and N.~Sun.
 \newblock Counting in two-spin models on $d$-regular graphs.
\newblock{\em  The Annals of Probability}, Vol.~42(6), 2383--2416, 2014.

\bibitem{Tal}
M.~Talagrand. The Parisi formula. {\em The
Annals  of  Mathematics}, Vol.~163(2), 221--263, 2006.


\bibitem{MJW}M.~J.~Wainwright, T.~S.~Jaakkola, and A.~S.~Willsky. 
\newblock A new class of upper bounds on the log partition function. 
\newblock {\em Information Theory, IEEE Transactions on}, Vol.~51(7), 2313--2335, 2005.


\bibitem{wu}
F.~Y.~Wu. 
\newblock The Potts model. 
\newblock {\em Reviews of modern physics}, Vol.~54(1), 235, 1982.


				
\end{thebibliography}
\end{document}